\documentclass[a4paper,10pt]{amsart}

\usepackage{amsmath,amsfonts,amsthm}

\usepackage{verbatim,textcomp}
\usepackage[pdftex]{graphicx,color}

\usepackage[latin1]{inputenc}
\usepackage[T1]{fontenc}
\usepackage{hyperref}
\usepackage{psfrag}

\newcommand{\Z}{\mathbb{Z}}
\newcommand{\R}{\mathbb{R}}
\newcommand{\C}{\mathbb{C}}
\newcommand{\CP}{\mathbb{C}P}
\newcommand{\RP}{\mathbb{R}P}
\newcommand{\T}{\mathbb{T}}
\newcommand{\D}{\mathcal{D}}
\newcommand{\F}{\mathcal{F}}
\newcommand{\Ve}{\text{Vert}}
\newcommand{\Ed}{\text{Edge}}

\DeclareMathOperator{\codim}{codim}

\theoremstyle{plain}
\newtheorem{thm}{Theorem}[section]
\newtheorem{lem}[thm]{Lemma}
\newtheorem{prop}[thm]{Proposition}
\newtheorem{coro}[thm]{Corollary}

\newtheorem{defn}[thm]{Definition}

\theoremstyle{remark}
\newtheorem*{rem}{Remark}
          
\theoremstyle{definition}
\newtheorem{exa}{Example}

\title{Enumeration of Real Conics and Maximal Configurations}

\date{\today}
\author{Erwan Brugallé}

\address{Université Pierre et Marie Curie,  Paris 6, 175 rue du Chevaleret, 75
013 Paris, France}

\email{brugalle@math.jussieu.fr}

\author{Nicolas Puignau}

\address{Universidade Federal do Rio de Janeiro, Ilha do Fundão, 21941-909 Rio de Janeiro, Brasil}

\email{puignau@im.ufrj.br}
\begin{document}

\begin{abstract}
We use floor decompositions of tropical curves to prove that any
enumerative problem concerning conics passing through
projective-linear subspaces in $\RP^n$ is maximal.
 That is, there exist generic configurations of real linear spaces such
 that 
all complex conics passing through these constraints are actually real.
\end{abstract}

\maketitle

\section{Introduction}\label{intro}
A rational curve of degree $d$ in $\C P^n$ is parameterized by a
polynomial map 
$$\begin{array}{cccc}
\phi :& \C P^1& \longrightarrow & \C P^n
\\ & [t,u]&\longmapsto & [P_0(t,u):\ldots :P_n(t,u)]
\end{array} $$
where the $P_i(t,u)$'s are homogeneous polynomials in two variables of
degree $d$
with no common factors. Since $Aut(\C P^1)$ has dimension 3 and 
all the  $P_i(t,u)$'s  are defined up to a common multiplicative
constant, the dimension of the space of
rational curves of degree $d$ in $\C P^n$ is
$$(n+1)(d+1) - 4 = (n+1)d + (n-3). $$

Consequently, if we are looking for rational curves in $\C P^n$ satisfying
exactly this number of independent conditions, we can reasonably 
expect the number of
solution to be finite. For example, if $L$ is a linear subspace of
codimension $j\ge 1$ in $\C
P^n$, the condition ``to intersect $L$'' imposes exactly $j-1$ independent
conditions on rational curves in $\C P^n$. 
Hence, if we choose a generic configuration 
$\omega=\{L_1,\ldots,L_\gamma\}$
of linear subspaces of 
$\C P^n$, with $l_j=\codim L_j\ge 1$, such that 
\begin{equation}\label{RR}\sum_{j=1}^{\gamma}(l_j-1)=(n+1)d
  +(n-3)\end{equation}
we expect a finite number of rational curves of degree $d$ in $\C P^n$
 intersecting all the linear subspaces in $\omega$. The term ``generic''
means that these linear subspaces have to be chosen so that they impose
altogether independent conditions.

It turns out that this number of rational curves, that we denote by
$N_{d,n}(l_1,\dots,l_{\gamma})$, is indeed finite and 
 doesn't depend on the configuration $\omega$ we
have chosen, but only on $n$, $d$, $l_1$,$\ldots$, $l_{\gamma}$. 
The numbers $N_{d,n}(l_1,\dots,l_{\gamma})$ are known as
\textit{Gromov-Witten invariants} of the projective space $\C
P^n$. 
For example, since there exists a unique line passing through two
distinct points in $\C P^n$, we have 
$$N_{1,n}(n,n)=1 \quad \forall n\ge 2. $$

For a more detailed introduction to Gromov-Witten theory, 
we refer the interested
reader to the excellent book \cite{Vain2}.
In this paper, it is convenient to
 extend the definition of the numbers
$N_{d,n}(l_1,\dots,l_{\gamma})$ to any set of $\gamma$ numbers in $\Z$
by 
$$N_{d,n}(l_1,\dots,l_{\gamma})=0\quad \text{if}\quad
\sum_{j=1}^{\gamma}(l_j-1)\ne(n+1)d 
  +(n-3) \quad \text{or}\quad \exists j, \ l_j\le 0\text{ or } l_j\ge n+1.$$

All linear spaces in our generic configuration $\omega$ can be chosen to be
real. In this case, it makes sense to enumerate real rational
curves in $\C P^n$ (i.e. rational curves which are invariant under the
complex conjugation of $\C P^n$) of degree $d$ intersecting our
configuration of real
linear subspaces. Unlike in the enumeration of complex curves, the
number of real solutions, denoted by $N^\R_{d,n}(l_1,\dots,l_{\gamma},\omega)$,
 now depends on the chosen configuration
$\omega$ of
linear spaces. 
Clearly, we have the inequality
$$N^\R_{d,n}(l_1,\dots,l_{\gamma},\omega)\le N_{d,n}(l_1,\dots,l_{\gamma})
  \quad \forall \omega.$$
However, it is unknown in general if there exists a real 
configuration $\omega$ such that all complex solutions are real.
For example, 
can the $92$ complex conics
passing through $8$ general lines in $\RP^3$  be real?
More generally,
it is an important and difficult
 question to ask how many solutions of an enumerative
problem can be real (see \cite[§7.2]{F}). When all complex solutions
can be real,
 we say that this enumerative problem is \textit{maximal}.

To stress how difficult these questions are, 
let us summarize the 
very few things  known in 2011 about the maximality of the enumerative
problems defined above. Since the corresponding Gromov-Witten
invariant is equal to 1, it is trivial that the problem is maximal in
the two following cases:
\begin{itemize}
\item $d=1$,  $l_j=n$ for some $j$;
\item $n=2$, $d=2$, and $l_1=l_2=l_3=l_4=l_5=2$.
\end{itemize}
It is also easy to see that the problem is maximal in the case $n=2$
and $d=3$ (and so $l_1=\ldots=l_8=2$). The first
systematic non-trivial result was obtained by Sottile who proved in
\cite{Sot3} that the problem is maximal as soon as $d=1$ (the so-called 
problems of "Schubert-type"). Recently, with the help of tropical
geometry, it was proved in \cite{Br7} that the problem is maximal for
$d=2$ and $n=3$. Up to our knowledge, nothing more was known before
our investigation.

\vspace{2ex}
Our main result is that the enumerative problems discussed above are
maximal when $d=2$. 
Theorem \ref{main} is a direct consequence of 
 Theorem \ref{corr}, Lemma \ref{real mod}, and Proposition \ref{conics}.
\begin{thm}\label{main}
For $n\ge 2$, $l_1\ge 1$, $\ldots$, $l_{\gamma}\ge 1$ satisfying
$$\sum_1^{\gamma} (l_j-1)=3n -1,$$
there exists a generic  configuration $\omega=\{L_1,\ldots,L_\gamma\}$ 
of real linear
subspaces of $\C P^n$ such that $\codim L_j=l_j$ and
$$N^\R_{2,n}(l_1,\dots,l_{\gamma},\omega)= N_{2,n}(l_1,\dots,l_{\gamma}).$$
\end{thm}
To prove Theorem \ref{main}, we use \textit{floor decomposition of
  tropical curves}. In his pioneer work \cite{Mik1}, Mikhalkin reduced
the enumeration of complex and real algebraic curves in $(\C^*)^2$ to
the enumeration of some piecewise linear graphs in $\R^2$ called 
plane tropical curves. Shortly after 
 these results
were extended in \cite{Mik08} and \cite{NS}
to the computation of genus 0 Gromov-Witten
invariants of projective spaces of arbitrary dimension. 
By stretching configurations of constraints along  some specific
direction,   Brugallé and Mikhalkin replaced
in \cite{Br6}
 the enumeration of tropical curves by a purely combinatorial study
 of their floor decompositions. As an application, they
 exhibited a generic configuration of 8 real lines in $\R P^3$
 with 92 real conics passing through them.

In this paper, we refine the technique used in
\cite{Br7} in the case of $\C P^3$
 to systematically study the case $d=2$. Along the way, we
will give a proof of Sottile's Theorem  different
from the original one.

\vspace{1ex}
The question of existence of non-trivial lower bounds for the numbers
$N^\R_{2,n}(l_1,\dots,l_{\gamma},\omega)$ is also a very important and
difficult problem about which not so much is known. The combination of
Welschinger invariants (see \cite{Wel1} and \cite{Wel2}) 
and tropical geometry allowed to exhibit such
non-trivial lower bounds in the case of rational curves passing
through points in $\R P^2$ or $\R P^3$, i.e. for $n=2$ or $3$ and $l_i=n$
$\forall i$ (see \cite{Wel1}, \cite{Mik1}, and \cite{IKS2} for the case
$n=2$, and \cite{Wel2}, and \cite{Br6} for the case $n=3$). 
In the case of enumeration of lines (and more generally in the
enumeration of real linear spaces), the existence of some non-trivial lower
bounds has been proved by Gabrielov and Eremenko in \cite{Gabi1}.
Up to our knowledge, the exact determination of the minimal value of 
$N^\R_{d,n}(l_1,\dots,l_{\gamma},\omega)$ (when non-trivial) is known
so far
only in the cases $n=2$, $d=3$ (\cite{DK}) and $d=4$ (\cite{Rey}), and
in the cases $d=1$ and $l_i=2$ for all $i$ (\cite{Gabi1}).

\vspace{1ex}
One could also study maximality of more general real enumerative
problems, for example by prescribing tangency conditions with constraints. 
We refer the interested reader to 
\cite{RoToVu}, \cite{Ber3}, \cite{Sot2}, and \cite{Br9}
for some partial answers in this direction.

\vspace{1ex}
We give in section \ref{section: geo trop} all tropical definitions
needed to prove Theorem \ref{main}. 
In particular, we set-up tropical enumerative
problems studied in this paper. Next, we explain 
in section \ref{section: floor dec}
 the main ideas of the
floor decomposition technique to solve these tropical enumerative
problems, before focusing
 on the easier particular cases of enumeration of lines and conics. 
Theorem \ref{main} is finally proved in section
\ref{section: proof}.

\vspace{2ex}
\textbf{Acknowledgment: } 
This work has been supported by the \textit{Réseau de Coopération France-Brésil}. E.B. is also
partially supported by the ANR-09-BLAN-0039-01.

\section{Tropical geometry}\label{section: geo trop}
In this section, we briefly review the tropical background needed in
this paper. For more details, we refer, for example, to 
\cite{Mik3}, \cite{St2}, \cite{Mik9}, and \cite{BIT}. 
\subsection{Rational tropical curves}
Given a finite graph $C$ (i.e. $C$ has a finite number of edges and
vertices) we denote 
by $\Ve^\infty(C)$ (resp. $\Ve^0(C)$) the set of its vertices 
 which are (resp. are not) $1$-valent, 
 and by 
$\Ed^\infty(C)$ (resp. $\Ed^0(C)$) the set of its edges  
 which are (resp. are not) adjacent to a $1$-valent vertex. 
Throughout the text, we will always assume that
the considered graphs \textbf{do not have any 2-valent vertices}.

\begin{defn}
A rational tropical curve $C$ is a  
 finite compact connected tree
equipped with a complete inner metric on
$C\setminus\Ve^\infty(C)$.
\end{defn}
By definition,  the 1-valent vertices of $C$ are 
at  infinite distance from all the other
points of $C$. 
Elements of $\Ed^\infty(C)$ are called the \textit{ends} of $C$. 
An edge in $\Ed^\infty(C)$ (resp. $\Ed^0(C)$) is said to be
\textit{unbounded} (resp. \textit{bounded}).

\begin{exa}
An example of rational tropical curve $C$ is depicted in Figure
\ref{exa1}. This curve has 5 unbounded edges, and 2 bounded
edges of finite
length
$a$ and $b$. The 
length of an edge of $C$ is written close to this edge in Figure \ref{exa1}.
\begin{figure}[h]
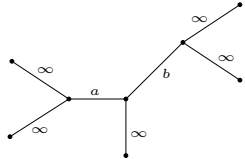
\caption{A rational tropical curve}\label{exa1}
\end{figure}
\end{exa}

Given $e$  an edge of a tropical curve $C$, we choose a point $p$ in
the interior of $e$ and
a unit vector $u_e$
of the tangent line to $C$ at $p$. Of course,  the vector $u_e$
depends on the choice of $p$ and is well-defined only up to
multiplication by $-1$, but this will not
matter in the following. We will sometimes need $u_e$ to have a
prescribed direction, and we will then precise this direction.
The
standard inclusion of $\Z^n$ in $\R^n$ induces a standard
inclusion of $\Z^n$ in the tangent space of $\R^n$  at any point of $\R^n$.

\begin{defn}
Let $C$ be a  rational tropical curve.
A
 continuous map  $f : C\setminus\Ve^\infty(C)\to \R^n$ is a
  tropical morphism if
\begin{itemize}
\item for any edge $e$ of $C$, the restriction $f_{|e}$ is a
  smooth map with $df(u_e)=w_{f,e}u_{f,e}$ where
 $u_{f,e}\in\Z^n$ is a
  primitive vector, and $w_{f,e}$ is a non-negative integer;

\item for any vertex $v$ in $\Ve^0(C)$ whose adjacent  edges are
  $e_1,\ldots,e_k$,
  one has the balancing condition
$$\sum_{i=1}^k  w_{f,e_i}u_{f,e_i}=0 $$
where $u_{e_i}$ is chosen so that it points away from $v$.
\end{itemize}
\end{defn}

The integer $w_{f,e}$
 is
called the \textit{weight of the edge $e$ with respect to $f$}.  When
no confusion is possible, we will  speak 
about the weight of an edge, without referring to the morphism $f$.
By abuse, we will
 denote $f:C\to\R^n$ instead of $f : C\setminus\Ve^\infty(C)\to \R^n$.
If 
$w_{f,e}=0$, we say that the morphism $f$ \textit{contracts} the edge
$e$. 
The morphism $f$ is called
\textit{minimal} if it does not contract
any edge.

Given $u=(u_1,\ldots,u_n)$ a vector in $\R^n$, we define
$$d_u= \max\{0, u_1,\ldots,u_n\}. $$
The \textit{degree} of a tropical morphism $f:C\to \R^n$ is defined by
$$\sum_{e\in\Ed^\infty(C)}w_{f,e}d_{u_{f,e}} $$
where $u_{f,e}$ is chosen so that it points to its adjacent 1-valent vertex.

We define the following vectors in $\R^n$:
$U_1=(-1,0,\ldots,0)$, $U_2=(0,-1,0,\ldots,0)$, $\ldots$, $U_n=(0,\ldots,0,-1)$,
  and $U_{n+1}=(1,\ldots,1)$. A tropical morphism $f:C\to\R^n$ of degree $d$ 
is said to be \textit{transverse at infinity} if $C$ has exactly
$(n+1)d$ non-contracted ends. Note that in this case, 
 for any $i=1,\ldots, n+1$ the curve $C$ has exactly $d$
  edges $e\in\Ed^\infty(C)$ with $u_{f,e}=U_i$, 
where $u_{f,e}$ is chosen so that it points to its adjacent 1-valent vertex.

\begin{exa}
In Figure \ref{exa2} are depicted a tropical conic in
$\R^2$ and a tropical conic in $\R^3$. For each unbounded edge $e$,
the vector $u_{f,e}$ pointing to infinity is written close to  $e$.
\begin{center}
\begin{figure}[h]
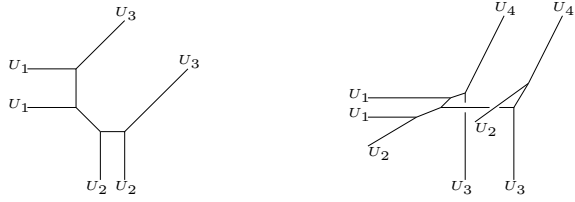
\caption{A tropical conic in $\R^2$ and in $\R^3$}\label{exa2}
\end{figure}
\end{center}
\end{exa}

Two tropical morphisms
$f_1:C_1\to\R^n$ and $f_2:C_2\to\R^n$ are said to be
\textit{isomorphic} if there exists
 an isomorphism of metric graphs $\phi:C_1\to
C_2$ such that $f_1=f_2\circ\phi$.
In this text, we consider  
 tropical curves and tropical morphisms up to
isomorphism. 

Two tropical morphisms $h:C_1\to \R^n$ and $h':C_2\to \R^n$ are said to be of
the same \textit{combinatorial type} if there exists a homeomorphism
of graphs $\phi:C_1\to C_2$  (i.e. we forget
about the metric on $C_1$ 
and $C_2$) such that $h=h'\circ \phi$,
 and 
$w_{h,e}=w_{h', \phi(e)}$ for all $e\in\Ed(C_1)$.

\vspace{1ex}
In this text, we need the notion of \textit{reducible} tropical
morphism.
Given $C_1$ and $C_2$  two tropical curves, $p_1$ and $p_2$ two points
respectively on $C_1$ (resp. $C_2$), the
topological gluing $C_1\cup_{(p_1,p_2)} C_2$ of $C_1$ and $C_2$ at $p_1$
and $p_2$ inherits 
naturally a structure of tropical curve from $C_1$ and $C_2$. The
curve $C_i$ can be seen as a subset of $C$, and the point $p_1=p_2$ is
called 
the \textit{node} of $C$.
\begin{defn}
A minimal tropical morphism $f:C\to \R^n$ is said to be reducible if
there exist two minimal tropical morphisms $f_1:C_1\to \R^n$ and 
 $f_2:C_2\to \R^n$, and a point $p_i\in C_i$, such that $C$ is the
gluing of $C_1$ and $C_2$ at the point $p_1$ and $p_2$, and 
$f_{|C_i}=f_i$.
\end{defn}
We denote such a reducible tropical morphism $f=f_1\cup_p
f_2:C_1\cup_p C_2:\to\R^n$.
If  $f:C\to \R^n$ is a reducible tropical morphism of degree $d$, then 
$f_i:C_i\to \R^n$ is a tropical morphism of degree $d_i$ and
$d_1+d_2=d$. In particular, if $d=2$ then $d_1=d_2=1$.

\subsection{Tropical linear spaces}
Defining tropical linear spaces of $\R^n$
in full generality would require much
more
material than needed in the rest of the paper. Moreover this would force us to
make the distinction between \textit{realizable} and
\textit{non-realizable} tropical linear spaces, notion that we want
to keep out of the scope of this note. Instead, we define a
restricted class of tropical linear spaces of $\R^n$, that we call
\textit{complete tropical linear spaces}.
For a general definition and 
study of tropical linear spaces, we refer to \cite{Spe3}.

\vspace{1ex}
Given $1\le i<j\le n+1$, we denote by $E_{i,j}$ the convex polyhedron of
$\R^n$ obtained by taking all non-negative
real linear combinations 
of all the vectors $U_k$ but $U_i$ and $U_j$, and we define 
$$H_0^n=\cup_{1 \le i<j\le n+1} E_{i,j}.$$ 

\begin{defn}\label{defi linear}
A tropical hyperplane of $\R^n$ is the translation of $H_0^n$ along
any vector of $\R^n$.

A complete tropical linear space of dimension $j$
is the  intersection of $n-j$ tropical hyperplane in general
position. The ambient  space $\R^n$ is a complete tropical linear space of dimension $n$.
\end{defn}
One could avoid the genericity assumption in Definition \ref{defi
  linear} by considering  tropical (or stable) intersections of tropical
hyperplanes in $\R^n$. We refer
to \cite{Mik3} or \cite{St2} for more details.

A tropical linear space of dimension $j$ is a finite polyhedral
complex of pure dimension $j$.

\begin{exa} A tropical plane and a tropical line in $\R^3$ are
  depicted in Figure \ref{exa3}.
\begin{center}
\begin{figure}[h]
\includegraphics[height=2.5cm]{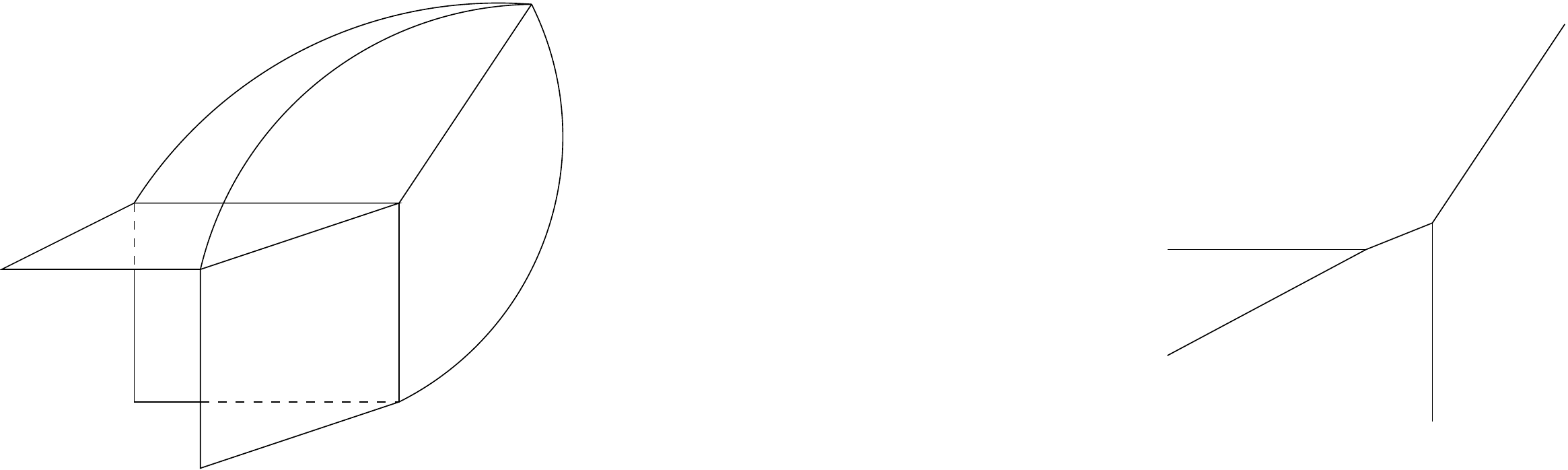}
\caption{A tropical plane and a tropical line in $\R^3$}\label{exa3}
\end{figure}
\end{center}
\end{exa}

\subsection{Tropical enumerative geometry}\label{TEG}
Now we have defined tropical rational curves and complete tropical linear
spaces in $\R^n$, we can play the same game as in section
\ref{intro}. Namely, let us fix some integers $d\ge 1$, $n\ge 2$,
$\gamma\ge 2$,
$l_1\ge 1$, $\ldots$, $l_{\gamma}\ge 1$ 
subject to equality (\ref{RR}),
and let us choose a configuration $\omega=\{L_1,\ldots,L_\gamma\}$ of 
complete  tropical linear spaces in $\R^n$ such that 
$\codim L_j=l_j$ for $j=1,\ldots,\gamma$.
Then we define $\T \mathcal C(\omega)$  as the set of all minimal rational 
morphisms $f:C\to \R^n$ of degree $d$ such that $f(C)$ intersects all
 tropical linear spaces in $\omega$.
This game is 
related to section \ref{intro} by the following fundamental Theorem.

\begin{thm}[Correspondence Theorem, \cite{Mik1}, \cite{Mik08},
    \cite{NS}]\label{corr}
If $\omega$ is generic, then the set  $\T \mathcal C(\omega)$ is
finite and composed of tropical morphisms transverse to infinity. 
Moreover, to each element $f$ in $\T \mathcal
C(\omega)$,
one can associate a positive integer number $\mu(f)$, called the
multiplicity of $f$, which depends only on $f$ and $\omega$ such that
$$N_{d,n}(l_1,\dots,l_{\gamma})=\sum_{f\in\T \mathcal C(\omega)}\mu(f). $$  
\end{thm}
There is a combinatorial definition of
the integer $\mu(f)$ just in terms of the tropical morphism
$f$ and the configuration $\omega$.
However we won't need it in this
paper, so instead of giving the precise definition of $\mu(f)$, let us
just explain its geometrical
meaning.

Theorem \ref{corr} is obtained by degenerating the standard complex
structure on $(\C^*)^n$ via the following self-diffeomorphism of  $(\C^*)^n$
$$\begin{array}{cccc}
H_t:&(\C^*)^n&\longrightarrow& (\C^*)^n
\\ &(z_i) &\longmapsto & (|z_i|^{\frac{1}{\log t}}\frac{z_i}{|z_i|})
 \end{array}$$
Namely, for any tropical complete linear space $L_j$ of codimension $l_j$
 in $\omega$, there exists a family
$(L_{t,j})_{t>0}$ of complex linear spaces in $(\C^*)^n$ of
 codimension $l_j$ such that the
 sets
$Log\circ H_t(L_{t,j})$ converges to $L_j$ when $t\to\infty$,
 in the Haussdorf metric on compact
 subsets of  $(\R^*)^n$. The map $Log$ is
 defined by $Log(z_i)=(\log|z_i|)$. Hence for each $t$, we associate a
 configuration $\omega_t=\{L_{t,1},\ldots,L_{t,\gamma}\}$ 
of linear subspaces of $(\C^*)^n$.
For  $t$ big enough, the configuration $\omega_t$ is generic if
$\omega$ is generic, so the
 complex rational curves of degree $d$ passing through
 all the linear spaces in $\omega_t$ form a finite set $\mathcal
 C(\omega_t)$. It turns out, and this is the core of Theorem
 \ref{corr},
 that 
the set $Log\circ H_t(\mathcal C(\omega_t))$ converges to the set $\T
\mathcal C(\omega)$, and that for any tropical morphism $f\in\T
\mathcal C(\omega)$ there exist exactly $\mu(f)$ complex curves in 
$\mathcal C(\omega_t)$ whose image under $Log\circ H_t$ converge to
$f(C)$.

Suppose now that the linear spaces $L_{t,j}$ are chosen to be all 
real (this is always
possible). In particular, all curves in $\mathcal C(\omega_t)$ are either real
or come in pairs of complex conjugated curves. 
A very important property of the map $H_t$ is that it commutes with
the standard complex conjugation in $(\C^*)^n$. As a consequence,
both curves in 
a pair of complex conjugated curves in $\mathcal C(\omega_t)$
have the same image under $Log\circ H_t$. 
In particular we have the following Lemma, where $[\mu(f)]_2$ denotes
the value modulo 2 of the integer $\mu(f)$.
\begin{lem}\label{real mod}
If $\omega$ is generic, then 
there exists a generic configuration $\Omega$ of real
linear spaces in $\R P^n$ 
such that there exist at least $\sum_{f\in\T\mathcal C(\omega)}[\mu(f)]_2$
 real rational curves of degree $d$ in $\R P^n$
intersecting all linear spaces in $\Omega$.
\end{lem}

Theorem \ref{main} is a  consequence of Theorem \ref{corr} and
Lemma \ref{real mod}: we  exhibit generic configurations $\omega$
such that the set $\T\mathcal C(\omega)$  contains exactly 
$N_{2,n}(l_1,\dots,l_{\gamma})$
distinct tropical curves. Hence, all of them must have multiplicity 1,
which implies the maximality of the corresponding enumerative
problem by Lemma  \ref{real
  mod}.
The main tool to exhibit such configurations $\omega$
is  the \textit{floor
decomposition} technique.

\subsection{Enumeration of tropical reducible conics}\label{red}
Here we state some easy facts about a small variation of the problem 
exposed in section \ref{TEG}. Namely, we enumerate tropical reducible conics
passing through a generic collection of complete tropical linear
spaces.

\vspace{1ex}
Next Lemma is  standard, see for example \cite{Br9} or \cite{GMK}.
\begin{lem}\label{dim red}
Let $\alpha$ be a combinatorial type of  reducible morphisms 
 $f_1\cup_p f_2:C_1\cup_p C_2\to \R^n$
of degree
2.
Then the space of all reducible tropical morphisms with combinatorial
type $\alpha$ is naturally a convex
polyhedron of 
 dimension at most $3n-2$, with equality if and only
if $C_i$ is trivalent and $p_i$ is not a vertex of $C_i$ for $i=1,2$.
\end{lem}

Let  us fix some integers 
$l_0\ge 0$ and $l^1_1$, $\ldots$, $l^1_{\gamma_1},l^2_1$, $\ldots$,
$l^2_{\gamma_2}\ge 1$ such that 
$$l_0+\sum_{i=1,2}\quad \sum_{j=1}^{\gamma_i}(l^i_j-1)=3n-2$$                                                                 
and let us choose $L_0$ a tropical complete linear space in $\R^n$
of codimension $l_0$ and
two configurations $\omega^1$ and $\omega^2$ of 
complete  tropical linear spaces in $\R^n$ such that 
$\omega^i=\{L^i_1,\ldots,L^i_{\gamma_i}\}$ with
 $\codim L^i_j=l^i_j$.
Then we define $\T \mathcal C_{red}(L_0,\omega^1,\omega^2)$  as the set of all
reducible rational 
morphisms 
$f=f_1\cup_p f_2:C_1\cup_p C_2\to \R^n$ 
of degree $2$ such that $f_i(C_i)$ intersects all
 tropical linear spaces in $\omega^i$ for $i=1,2$ and $f(p)\in L_0$.
Note that if $\T \mathcal C_{red}(L_0,\omega^1,\omega^2)\ne
\emptyset$, we necessarily have
$$\sum_{j=1}^{\gamma_i}(l^i_j-1)\le 2n-2 \quad \text{for}\quad i=1,2.$$

Next Lemma is a straightforward application of Lemma \ref{dim red} and 
standard techniques in
tropical enumerative geometry, see for example \cite{Br9}, \cite{NS},
or
 \cite{GMK}.

\begin{lem}\label{generic red}
For a generic choice of  $L_0$, $\omega^1$, and $\omega^2$, the set 
$\T \mathcal C_{red}(L_0,\omega^1,\omega^2)$ is finite, and composed
of tropical morphisms transverse to infinity.
\end{lem}

\vspace{1ex}
Note that we can pose the same problem in complex geometry, and that
we can easily give the answer in terms of the numbers
$N_{1,n}$. Namely, let
 $L_0$ a  linear space in $\C P^n$ of codimension $l_0$ and
two configurations $\omega^1$ and $\omega^2$ of 
 linear spaces in $\C P^n$ such that 
$\omega^i=\{L^i_1,\ldots,L^i_{\gamma_j}\}$ with
 $\codim L^i_j=l^i_j$.
Then we define 
$N^{red}_{2,n}(l_0,\{l^1_1,\ldots,l^1_{\gamma_1}\},\{l^2_1,\ldots,l^2_{\gamma_1}\})$
as the number of reducible  
 conics $C_1\cup_p C_2$ in $\C P^n$
 such that $C_i$ intersects all
  spaces in $\omega^i$ for $i=1,2$ and $p\in L_0$.
\begin{lem}\label{complex red}
With the hypothesis above, we have
$$N^{red}_{2,n}(l_0,\{l^1_1,\ldots,l^1_{\gamma_1}\},\{l^2_1,\ldots,l^2_{\gamma_2}\})
=\prod_{i=1,2}N_{1,n}\left(2n-1
-\sum_{j=1}^{\gamma_i}(l^i_j-1), l^i_1,\ldots,l^i_{\gamma_i} \right).$$ 
\end{lem}
\begin{proof}
Let $V_i$ be the  algebraic variety in $\CP^n$ given by the
union of all lines passing through  all linear spaces in
$\omega^i$. By definition, the number 
$N^{red}_{2,n}(l_0,\{l^1_1,\ldots,l^1_{\gamma_1}\},\{l^2_1,\ldots,l^2_{\gamma_1}\})$
is equal to the number of intersection points in $V_1\cap V_2\cap
L_0$, i.e. is equal to the product of the degree of $V_1$ and $V_2$. 
Since the configuration of linear spaces is generic, $V_i$
has dimension $2n-1-\sum_{j=1}^{\gamma_i}(l^i_j-1)$. The degree of
$V_i$ is its number of intersection points with a generic linear space in $\C
P^n$ of complementary dimension, 
and this number is by definition  precisely 
$N_{1,n}\left( 2n-1
-\sum_{j=1}^{\gamma_i}(l^i_j-1), l^i_1,\ldots,l^i_{\gamma_i}\right)$.
\end{proof}

\section{Floor decomposition of tropical curves}\label{section: floor dec}
Here we explain how to enumerate complex curves of degree $1$ and
$2$ with the help of
 the floor decomposition technique. This technique works for any degree
 (see \cite{Br7}, \cite{Br6b}, \cite{Br6}) but the exposition of the
 method in its full generality would require a quite heavy formalism which 
in our opinion would harm to the well understanding of this
text. Hence we  just give the main idea of the method before
focusing on the degree 1 and 2 cases. 

Note that  floor decomposition
technique has strong connections with the Caporaso and Harris method
(see \cite{CapHar1}),  extended later by Vakil (see \cite{Vak1}), and
with the neck-stretching method in symplectic field theory (see \cite{EGH}).

We denote by $\pi:\R^n\to \R^{n-1}$  the linear projection
forgetting the last coordinate. 
Given a minimal tropical morphism $f:C\to\R^n$, the morphism $\pi\circ
f:C\to\R^{n-1 }$ is not minimal in general. However, there exists a
unique tropical curve $C'$ equipped with a map $\rho :C\to C'$  
and a unique minimal tropical morphism $f':C'\to\R^{n-1}$
such that $f=f'\circ \rho$. We say that $f'$ is \textit{induced} by
$\pi\circ f$, and that $f$ is a \textit{lifting} of $f'$.

\subsection{General method}
The starting idea of floor decomposition is to compute the numbers 
$N_{d,n}(l_1,\dots,l_{\gamma})$ by induction on the dimension $n$. 
As easy at it sounds,
 this approach does not work straightforwardly and one has to 
work carefully: it is easy to compute that through one point $p$ and
two tropical lines $L_1$ and $L_2$
 in $\R^3$ passes exactly 1 tropical line $L$ 
 (see example \ref{exa6}). However, there exists
 infinitely many tropical lines in $\R^2$ passing trough $\pi(p)$,
 $\pi(L_1)$, and $\pi(L_2)$, and without knowing $L$, it is not clear at all which one
of these planar lines is $\pi(L)$.

To make the induction work, we first stretch the configuration
$\omega$ is the direction $U_n=(0,\ldots, 0,-1)$. Then the tropical curves
we are counting break in several \textit{floors} for which we can apply induction.

\begin{exa}\label{exa6}
Let us explain how to use the floor decomposition technique in a simple
case. Let use choose a point $p$ and two tropical lines $L_1$ and
$L_2$ in $\R^3$ such that $L_1$ (resp. $L_2$) consists of only one edge with
direction $(0,1,0)$ (resp. $(1,0,0)$) contained in the horizontal
plane with equation ${z=a_1}$ (resp. ${z=a_2}$). If the third coordinate of
$p$ is much more bigger than $a_1$ which in its turn  is  much more bigger
than $a_2$, then the unique tropical line $L$ in $\R^3$ passing through
$p$, $L_1$, and $L_2$ is depicted in Figure \ref{fig6}, and $\pi(L)$
is the unique tropical line in $\R^2$ passing through $\pi(p)$ and
$\pi(L_1)\cap \pi(L_2)$.

\begin{figure}[h]
\includegraphics[height=2cm, angle=0]{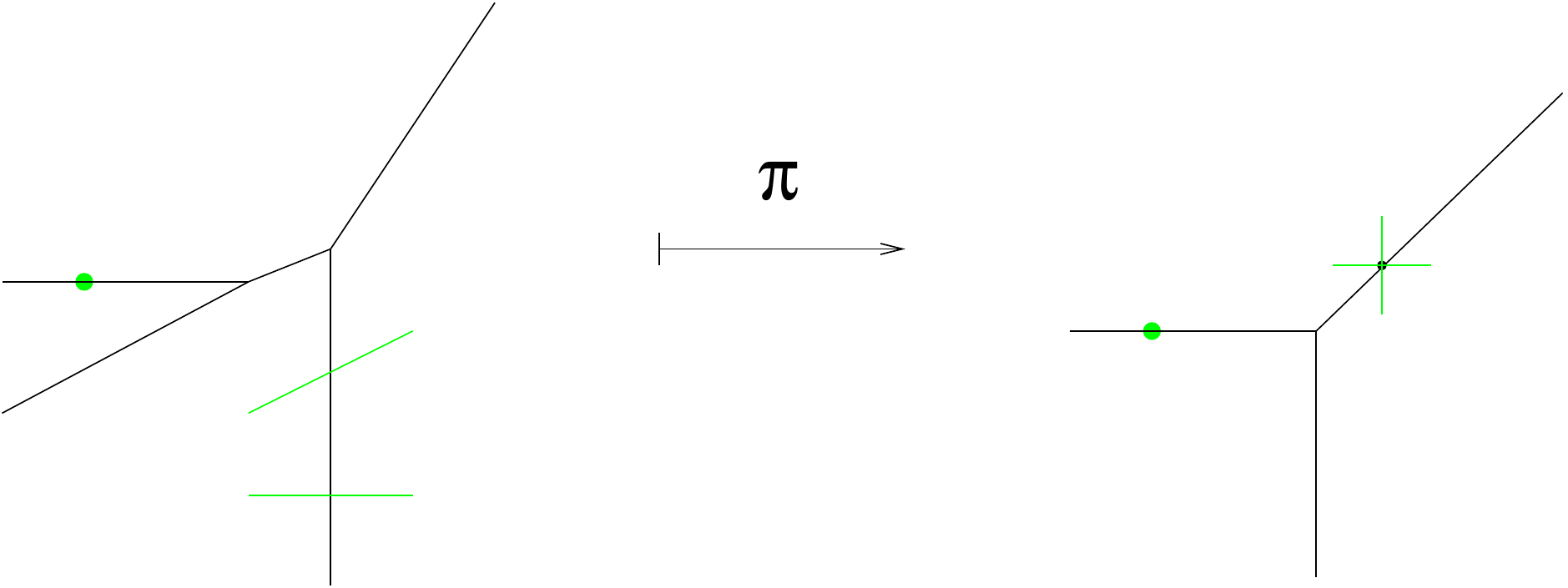}
\caption{Floor decomposition technique to compute $N_{1,3}(3,2,2)=1$}\label{fig6}
\end{figure}
\end{exa}

\begin{defn}\label{floor}
Let $C$ be a rational tropical curve and
$f:C\to \R^n$ a tropical morphism. An elevator of $f$ is an edge $e$ 
of $C$ with $u_{f,e}=(0,\ldots, 0,\pm 1)$.
A floor of $f$ is a connected
component of $C$ minus all its elevators.
\end{defn}
Note that if $f$ is a morphism of degree $d$ in $\R^n$ and $\F$ is a
floor of $f$, then $C$ induces a structure of tropical curve
on $\F$ and
$\pi\circ f_{|\F}:\F\to\R^{n-1}$ is a tropical morphism of degree
$1\le d'\le d$. The integer $d'$ is called the \textit{degree} of $\F$.

\begin{exa}
Examples of planar and spatial conics are depicted in
Figure \ref{confloors}. Elevators are depicted in dotted lines.
\begin{figure}[h]
\begin{tabular}{ccc}
\includegraphics[height=2cm, angle=0]{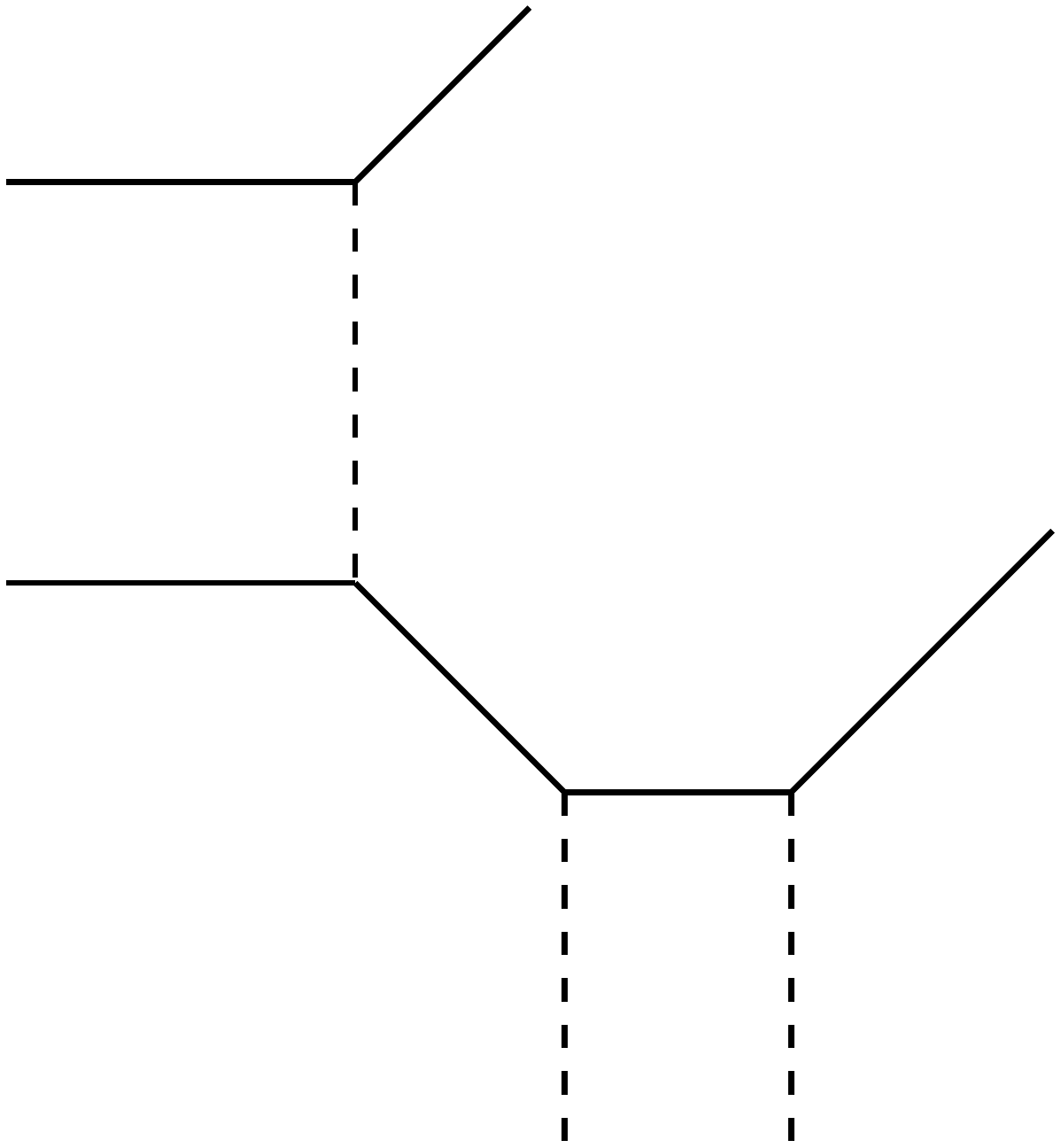}&
\includegraphics[height=2cm, angle=0]{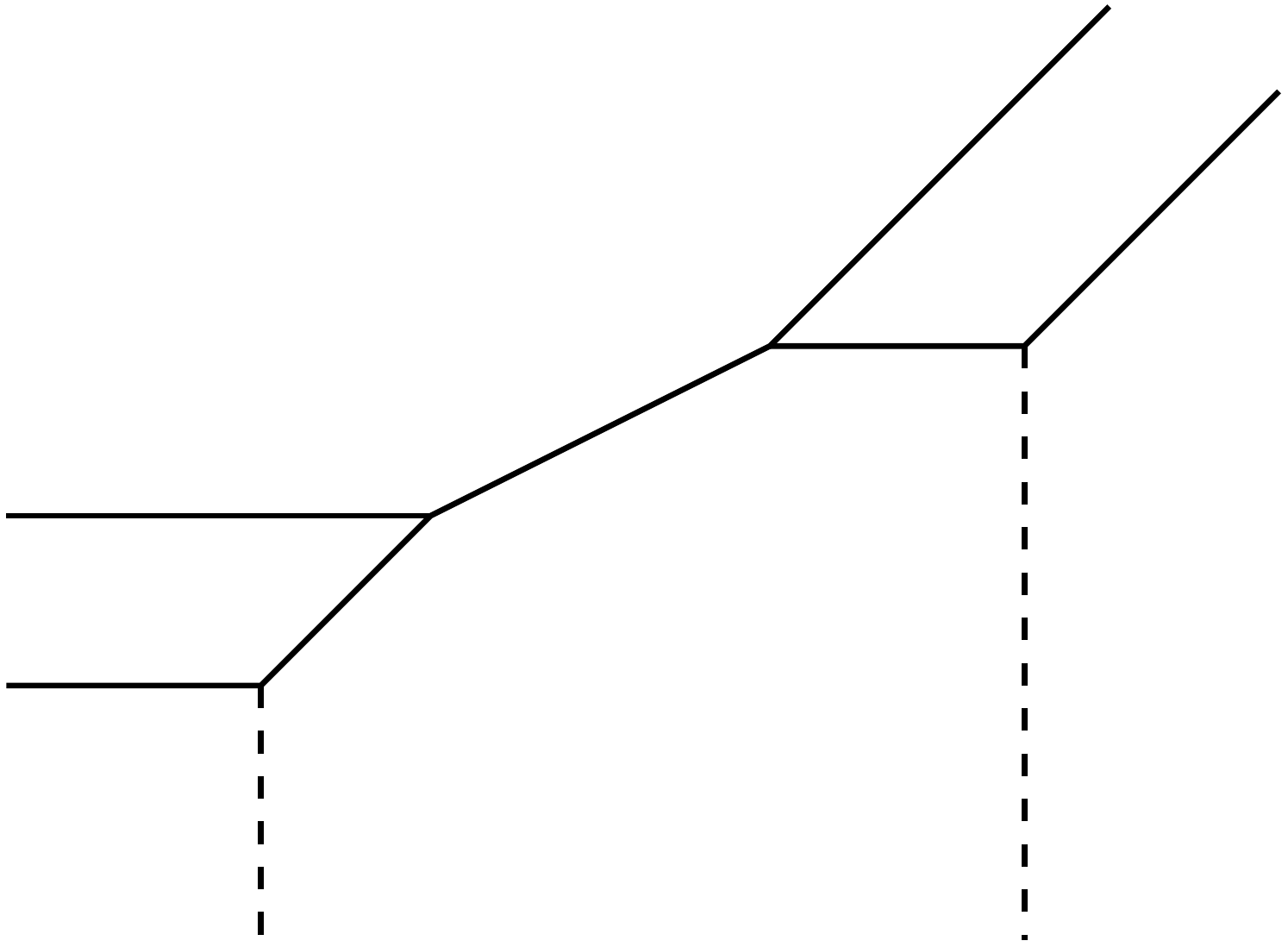}&
\includegraphics[height=2.5cm, angle=0]{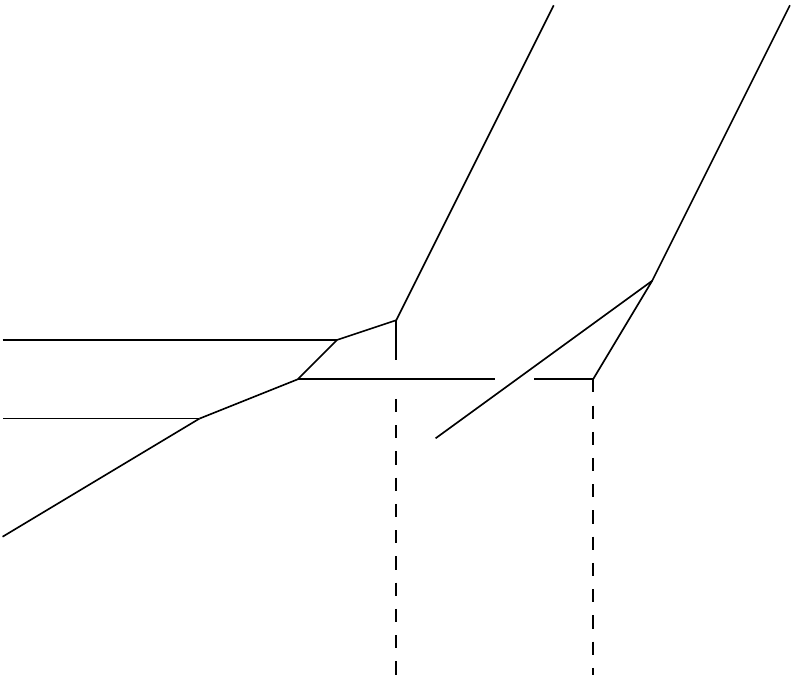}\\
a) A planar conic with two floors&
b) A planar conic
with one floor&
c) A spatial conic with one floor 
\end{tabular}
\caption{}\label{confloors}
\end{figure}
\end{exa}

Definition \ref{floor} extends to any tropical varieties in
$\R^n$, however we keep on restricting ourselves to 
 complete tropical linear spaces.
\begin{defn}\label{wall}
Let $L$ be a complete tropical linear space in $\R^n$ of codimension
$j$. 
The wall
(resp. floor) of $L$
is the union of all faces of $L$ of codimension $j$
which contain (resp. do not contain)
the direction 
$(0,\ldots, 0,1)$. 
\end{defn}
Note that if $L$ is a complete tropical linear space of codimension $j$
in $\R^n$ with
wall $W$ and floor $F$, then $\pi(W)$ is a complete tropical 
linear space of codimension
$j$ in $\R^{n-1}$, and  $\pi(F)=\pi(L)$ is a complete tropical
linear space of codimension
$j-1$ in $\R^{n-1}$.

\begin{exa}
A tropical plane $L$ in $\R^3$ together with its wall $W$ are depicted in
Figure \ref{exa5}.
Clearly,  $\pi(L)=\R^2$ and $\pi(W)$ is a tropical line in $\R^2$.
\begin{center}
\begin{figure}[h]
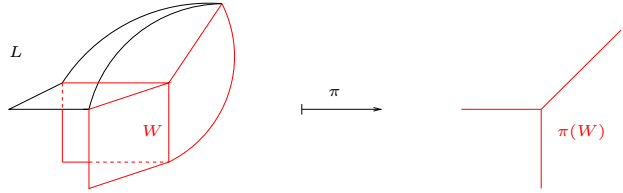
\caption{$\pi(L)=\R^2$ and $\pi(W)$ is a tropical line}\label{exa5}
\end{figure}
\end{center}
\end{exa}

Let us fix some integers $d\ge 1$, $n\ge 2$,
$\gamma\ge 2$,
$l_1\ge 1$, $\ldots$, $l_{\gamma}\ge 1$ 
subject to equality (\ref{RR}),
and let us choose a generic configuration $\omega=\{L_1,\ldots,L_\gamma\}$ of 
complete  tropical linear spaces in $\R^n$ such that 
$\codim L_j=l_j$ for $j=1,\ldots,\gamma$. As before,
$\T \mathcal C(\omega)$  is the set of all rational tropical 
morphisms $f:C\to \R^n$ of degree $d$ such that $f(C)$ intersects all
 tropical linear spaces in $\omega$.

\begin{defn}
An element $L$ of $\omega$ is called a vertical (resp. horizontal)
constraint for $f\in\T
\mathcal C(\omega)$ if $f(C)\cap L$ lies in the wall (resp. floor) of
$L$.
\end{defn}

Let us denote by $\Ve(L_j)$ the set of vertices
of the complete tropical linear space $L_j$,
and let us fix a hypercube $\mathcal H_{n-1}$ in $\R^{n-1}$ 
 such that the cylinder $\mathcal H_{n-1}\times \R$ contains the set
 $\cup_{j=1}^\gamma\Ve(L_j)$.
Given two points $v=(v_1,\ldots,v_n)$ and $w=(w_1,\ldots,w_n)$ in
$\R^n$, we define $|v-w|_n=|v_n - w_n|$.
Finally, we define  $R_{\mathcal H}$ to be the length of the edges of
$\mathcal H_{n-1}$, and 
$$R(\omega)=\min_{j\ne k, \ v\in\Ve(L_j), \ w\in\Ve(L_k)}|v-w|_n. $$
The following observation is the key
point of the technique.

\begin{prop}[Brugallé-Mikhalkin \cite{Br7}, \cite{Br6b},
    \cite{Br6}]\label{floor horizontal} 
There exists a real number $D(n,d)$, depending only in $n$ and $d$,
such that  if $R(\omega)\ge R_{\mathcal H}D(n,d)$ 
 then
 for each morphism $f:C\to\R^n$ in $\T \mathcal C(\omega)$ and for
each floor $\F$ of $f$, $f(\F)$ meets one and exactly one horizontal
constraints.
\end{prop}

\begin{defn}
If $R(\omega)\ge R_{\mathcal H}D(n,d)$, we say that $\omega$ is a \textit{$(d,n)$-decomposing configuration}.
\end{defn}
Note that to be a $(d,n)$-decomposing configuration is the same than
imposing conditions
on the relative position of the vertices of elements of $\omega$. In
particular, it makes sense to say that a configuration $\omega$ is
$(d,n)$-decomposing even if its elements do not satisfy equality (\ref{RR}).
\vspace{1ex}

The choice of the preferred direction $U_n$ provides a natural partial order among the tropical linear
spaces.
\begin{defn}
Let $L$ and $L'$ in $\R^n$ be two complete tropical linear spaces. 
We say that $L$ is
\textit{higher} than $L'$, and denote $L \gg L'$, if any vertex of $L$
has greater last coordinate than all vertices of $L'$. 
\end{defn}

Note that $\R^n$ is greater than any other tropical linear space.
Before explaining in detail the case of lines and conics, we need to
introduce a notation. Given a generic configuration $\{L_1,\ldots,L_\gamma\}$
of complete tropical linear spaces in $\R^{n}$, $k\in\{1,\ldots,\gamma \}$, 
and $A\subset\{1,\ldots,\gamma \}$, we define the following 
complete tropical linear spaces in $\R^{n-1}$
$$L'_k=\pi(W_k),\quad
\widehat L'_k=\pi(L_k),  \quad
\text{and}\quad\widetilde L'_A=\bigcap_{j\in A}\pi(L_j)$$
where $W_j$ is the wall of $L_j$. We also denote by
$\widetilde L'_k$ the complete tropical linear space 
$\widetilde L'_{\{1,\ldots,k\}}$.

\subsection{The case $d=1$}\label{d=1}
Suppose that $d=1$, so that 
$\sum_{j=1}^{\gamma}(l_j-1)=2n-2$. We choose a $(1,n)$-decomposing configuration
$\omega=\{L_1,\ldots,L_\gamma\}$ of complete tropical linear spaces
in $\R^n$, $n\ge 3$, such that $L_j$ has codimension $l_j$ and $L_{j+1}$ is
higher than $L_j$ for $j=1,\ldots,\gamma-1$. We denote by $W_k$ the wall
of the constraint $L_k$.
We denote by $\T \mathcal C(\omega)^{(k)}$ the subset of 
tropical morphisms in 
$\T \mathcal C(\omega)$ whose floor meets the horizontal constraint
$L_k$ (remember that in degree one, the floor is unique).

According to Proposition \ref{floor horizontal}, we have
$$\T \mathcal C(\omega)=\bigsqcup_{k=1}^{\gamma}  \T \mathcal C(\omega)^{(k)}.$$
Given $k$ in $\{1,\ldots,\gamma\}$, we define $\omega'^{(k)}=\{\widetilde L'_{k-1},\widehat L'_k,
L'_{k+1},\ldots, L'_\gamma\}$.
Since $\omega$ is generic,
the tropical linear space $L'_k$ has codimension $l_k$, 
$\widehat L'_{k}$ has codimension $l_k-1$, 
and $\widetilde L'_{k-1}$ has codimension $\sum_{j=1}^{k-1} (l_j-1)$ if non-empty.
If $f:C\to\R^n$ is an element of $\T \mathcal C(\omega)^{(k)}$, then 
the tropical morphism $\pi\circ f:C\to\R^n $ induces obviously an element of 
$\T \mathcal C(\omega'^{(k)})$. Moreover, any tropical
morphism $f'$ in $\T \mathcal C(\omega'^{(k)})$ can be lifted in a unique 
way as 
an element $f$ of 
 $\T \mathcal C(\omega)^{(k)}$: the only unknown is the location of
the elevator of $f$, which is given by the unique intersection
point of $f'(C')$ with $\widetilde L'_{k-1}$ (see Examples \ref{exa6} and
 \ref{exa7}).
Hence, there exists a natural bijection 
between the two sets $\T \mathcal C(\omega)^{(k)}$ and $\T \mathcal
C(\omega'^{(k)})$. Moreover, one can show that this bijection respects
 multiplicity of tropical curves. In other words, we have the
 following Proposition

\begin{prop}[Brugallé-Mikhalkin \cite{Br7}, 
\cite{Br6b}, \cite{Br6}]\label{floor d=1}
For any $k$ in $\{1,\ldots,\gamma\}$, we have
$$\sum_{f\in \T \mathcal C(\omega)^{(k)}}\mu(f)=\sum_{f'\in \T
  \mathcal C(\omega'^{(k)})}\mu(f').$$
\end{prop}
 Note that
$\T \mathcal
C(\omega'^{(k)})=\emptyset$ if $k=1$ or  $\sum_{j=1}^{k-1} (l_j-1)>n$.
Proposition \ref{floor d=1} allows one to
 compute all the numbers $N_{1,n}$ out of the numbers
$N_{1,n-1}$. Since it is trivial that $N_{1,2}(2)=1$, all 
the numbers $N_{1,n}$ can be computed using Proposition \ref{floor d=1}.

\begin{exa}\label{exa7}
Let $\omega=\{L_1,L_2,L_3,L_4\}$ be a $(1,3)$-decomposing configuration of 4
lines in $\R^3$ with $L_{j+1}\gg L_{j}$.  The set $\T \mathcal C(\omega)^{(k)}$ is non-empty
only for $k=2,3$, and
the corresponding projected configurations  in $\R^2$ are 
$\omega'^{(2)}=\{\pi(L_1), \pi(L_2),\pi(W_3),\pi(W_4)\}$ and
$\omega'^{(3)}=\{\pi(L_1)\cap\pi(L_2), \pi(L_3),\pi(W_4)\}$
(see Figure \ref{fig7}a).
Since there exists only one line passing through two points in the
plane, we get that $N_{1,3}(2,2,2,2)=1+1=2$.

To depict tropical curves passing through a decomposing configuration,
we use the following convention: a floor (resp. elevator) of the curve
is represented by an ellipse (resp. a vertical edge); a constraint
intersecting a floor (resp. elevator) is depicted by a dotted segment
intersecting the corresponding ellipse (resp. vertical edge); a
constraint is represented by a horizontal (resp. vertical) segment if
it is a horizontal (resp. vertical) constraint for the curve.    

For example, the two tropical lines passing through the four lines
$L_1,\ldots, L_4$ are represented by the  diagrams depicted in
Figures \ref{fig7}b and c.                                                                                                                                                                                                                                                                                                          
\begin{figure}[h]
\begin{tabular}{cccc}
\includegraphics[height=4cm, angle=0]{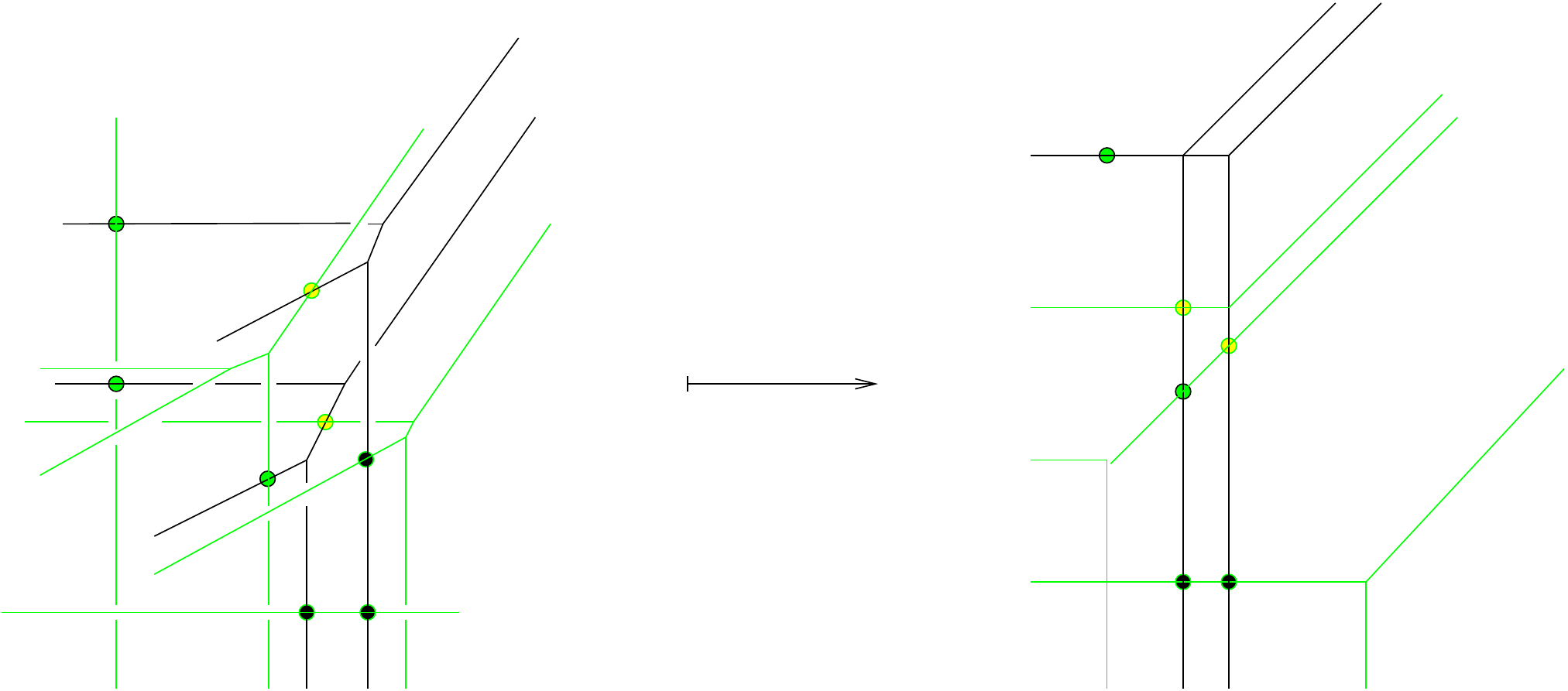}&
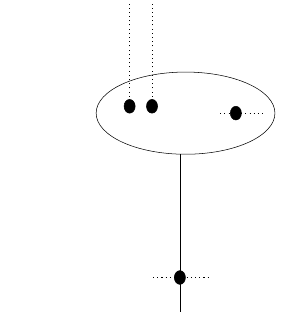& \hspace{3ex} &
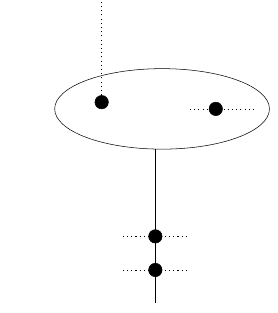

\\ \\ a) & b) Line in $ \T \mathcal C(\omega)^{(2)}$ &&  c) Line in $ \T
\mathcal C(\omega)^{(3)}$
\end{tabular}
\caption{Floor decomposition technique to compute $N_{1,3}(2,2,2,2)=2$}\label{fig7}
\end{figure}
\end{exa}

\begin{coro}\label{rec line}
Given $l_1,\ldots,l_\gamma\ge 1$ such that $\sum_{j=1}^\gamma(l_j-1)=2n-2$, we have 
$$N_{1,n}(l_1,\ldots,l_\gamma)=\sum_{k=2}^{\gamma}
N_{1,n-1}\left(\sum_{i=1}^{k-1} (l_i-1),l_k-1,l_{k+1},\ldots,l_{\gamma}\right). $$
\end{coro}

\begin{exa}
Let us  compute the numbers
$C(n,l)=N_{1,n}(l,l_1, \ldots,l_{2n-1-l})$ with $2\le l\le n$ and 
$l_1=\ldots=l_{2n-1-l}=2$.  According to 
 Corollary \ref{rec line},  for any $n\ge 2$ and
$3\le l\le n$ we get 
\begin{equation}\label{catalan}
C(n,l)=C(n-1,l-1) + C(n,l+1).
\end{equation}
Hence we can extend the definition of the numbers 
$C(n,l)$ for all pairs $(n,l)$ with $n\ge 1$ according to
relation (\ref{catalan}), which is a Pascal type relation. The sequence
$$A(n,l)=\left(\begin{array}{c} 2n- l -1 \\ n-1\end{array} \right) -
\left(\begin{array}{c} 2n- l -1 \\ n\end{array} \right)$$
also satisfies relation (\ref{catalan}), and we have
$$\forall n\ge 1 \quad C(n,0)=A(n,0)=0\quad \text{and}\quad
C(n,n)=A(n,n)=1 $$
so these two sequences must be equal on the set 
$\{(n,l)\in\Z^2\ | \  n\ge 1\}$. Hence we get
\begin{equation}\label{rec}
 \forall n\ge 2,\quad \forall l\in\{2,\ldots,n\}, \quad
C(n,l)= \left(\begin{array}{c} 2n- l -1 \\ n-1\end{array} \right) -
\left(\begin{array}{c} 2n- l -1 \\ n\end{array} \right).
\end{equation}
In particular, we find again the Catalan numbers
$$C(n,2)= C(n,1)=\frac{1}{n}\left(\begin{array}{c} 2n- 2
  \\ n-1\end{array} \right) .$$
Note that Corollary \ref{rec line} and the relation \ref{rec}
almost immediatly imply that
$$\forall n\ge 2,\quad \forall k,l\in\{2,\ldots,n\}, \quad
N_{1,n}(k,l,l_1, \ldots,l_{2n-k-l})= 
\left(\begin{array}{c} 2n- l -k \\ n-k\end{array} \right) -
\left(\begin{array}{c} 2n- l -k \\ n\end{array} \right) $$
where $l_1=\ldots=l_{2n-k-l}=2$. 
\end{exa}

\subsection{The case $d=2$}\label{d=2}
Let us suppose that $d=2$, so that 
$\sum_{j=1}^{\gamma}(l_j-1)=3n-1$. We choose a $(2,n)$-decomposing configuration
$\omega=\{L_1,\ldots,L_\gamma\}$ of complete tropical linear spaces
in $\R^n$, $n\ge 3$, such that $L_j$ has codimension $l_j\ge 2$ and $L_{j+1}$ is
higher than $L_j$ for $j=1,\ldots,\gamma-1$. As in section \ref{d=1}, we
denote by $W_k$ the wall
of the constraint $L_k$.
Given $f$ an element of $\T \mathcal C(\omega)$,
either $f$ has one floor of degree 2, or it has 2 floors of degree
1.

\vspace{1ex}
Let us first deal with the case of a conic with one floor of degree 2.
Let $k$ be an integer in $\{1,\ldots,\gamma\}$, and $A\sqcup B$ a
partition of $\{1,\ldots,k-1\}$ into two 
sets. 
We denote by $\T \mathcal C(\omega)^{(k,A,B)}$ the set of 
all tropical morphisms in 
$\T \mathcal C(\omega)$ with one floor $\F$ of degree 2  such that\\
\begin{tabular}{p{0.8\textwidth}p{0.2\textwidth}}
\begin{itemize}
\item the floor $\F$ meets the
 horizontal constraint
$L_k$;
\item one of the two elevators of $f$ meets all the constraints $L_j$
  with $j\in A$, while the other elevator meets all the constraints $L_j$
  with $j\in B$.
\end{itemize} & \begin{center}
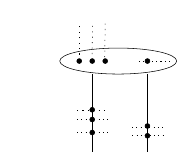\end{center}
\end{tabular}

According to Proposition \ref{floor horizontal}, the set of all
tropical morphisms in $\T \mathcal C(\omega)$ with a single floor is
equal to
$$\bigsqcup_{k=1}^{\gamma}\quad \bigsqcup_{A\sqcup B=\{1,\ldots,k-1\}}  \T
\mathcal C(\omega)^{(k,A,B)}.$$
We define $\omega'^{(k,A,B)}=\{\widetilde L'_A,\widetilde L'_B,\widehat L'_k, L'_{k+1},\ldots, L'_\gamma\}$.
The complete tropical linear space $L'_k$ has codimension $l_k$, 
the space $\widehat L'_k$ has codimension $l_k-1$, 
and $\widetilde L'_A$ (resp. $\widetilde L'_B$ )
has codimension $\sum_{j\in A} (l_j-1)$ (resp. $\sum_{j\in B}
(l_j-1)$)
if non-empty.
As in section \ref{d=1}, there is natural map
$$\phi_{(k,A,B)}:\T \mathcal C(\omega)^{(k,A,B)}\to 
\T \mathcal C(\omega'^{(k,A,B)}).$$
Contrary to section \ref{d=1}, the map $\phi_{(k,A,B)}$ is injective and
respects the multiplicity if and only if 
$$\codim \widetilde L'_A\ge 2,\quad 
\codim \widetilde L'_B\ge 2,\quad \text{and}\quad  
\codim \widehat L'_k\ge 2.$$

In general, given $f'\in \T \mathcal C(\omega'^{(k,A,B)})$ we have
$$\sum_{f\in\phi_{(k,A,B)}^{-1}(f')}\mu(f)=2^{m_{(k,A,B)}}\mu(f'). $$
where $m_{(k,A,B)}$ is the number of spaces in 
$\{\widetilde L'_A,\widetilde L'_B, \widehat L'_k\}$ of
codimension 1. The factor $2^{m_{(k,A,B)}}$ is just the manifestation
of the fact  that a conic in
the projective space intersect a hyperplane into two points (counted
with multiplicity).
Altogether we hence have the following Proposition.

\begin{prop}[Brugallé-Mikhalkin \cite{Br7}, 
\cite{Br6b}, \cite{Br6}]\label{floor d=2 1}
Given any $k,A,$ and $B$ as above, we have
$$\sum_{f\in \T \mathcal C(\omega)^{(k,A,B)}} \mu(f) =
2^{m_{(k,A,B)}}\sum_{f'\in \T \mathcal C(\omega'^{(k,A,B)})}\mu(f').$$
\end{prop}

\vspace{1ex}
We treat now the case of tropical morphisms $f$ in 
$\T \mathcal C(\omega)$ with two floors of degree 1. 
Let $k_2<k_1$ be two integers in $\{1,\ldots,\gamma\}$, 
$A \sqcup B$ be a partition of $\{1,\ldots, k_1-1\}$, $D$
a subset of $\{k_1+1,\ldots,k_2-1\}$,
and $C_1\sqcup C_2$ be a partition of $\{k_1+1,\ldots,\gamma\}\setminus\left(\{k_2\}\cup D\right)$.

We denote by $\T \mathcal C(\omega)^{(k_1,k_2,A,B,C_1,C_2,D)}$ the set of 
 tropical morphisms in 
$\T \mathcal C(\omega)$ with two floors $\F_1$ and $\F_2$ of degree 1
  such that\\
\begin{tabular}{p{0.8\textwidth}p{0.2\textwidth}}
\begin{itemize}
\item the floor $\F_i$ meets the constraint  
$L_{k_i}$, and all the constraints $L_j$ with $j\in C_i$;
\item one of the two unbounded elevators of $f$ meets all the constraints $L_j$
  with $j\in A$, while the other unbounded 
elevator meets all the constraints $L_j$
  with $j\in B$;
\item the bounded elevator of $f$ meets all the constraints $L_j$
  with $j\in D$.\end{itemize} & \begin{center}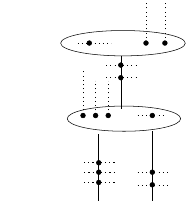\end{center}
\end{tabular}

According to Proposition \ref{floor horizontal}, the set of all
tropical morphisms in $\T \mathcal C(\omega)$ with  two floors is
equal to
$$\bigsqcup_{1\le k_1<k_2\le \gamma}\quad \bigsqcup_{A,B,C_1,C_2,D}  \T
\mathcal C(\omega)^{(k_1,k_2,A,B,C_1,C_2,D)}.$$
Given $i=1,2$, we define the following integers
$$\begin{array}{lll}
l'^i_j=l_j\ \text{if }j\in C_i \quad &
l'_A=\sum_{j\in A}(l_j-1) \quad & 
l'^1_0=2n-6 -  \sum_{j\in C_1}(l'_j-1) -l'_A -
l'_B-l'_{k_1}
\\ \\ l'_{k_i}= l_{k_i}-1
&   l'_B=\sum_{j\in B}(l_j-1)
 &
l'^2_0=2n-4 -  \sum_{j \in C_2}(l'_j-1) -l'_{k_2}
\end{array}$$
Finally, we denote by $\gamma_i$ the cardinal of the set $C_i$.

\begin{prop}[Brugallé-Mikhalkin \cite{Br7}, 
\cite{Br6b}, \cite{Br6}]\label{floor d=2 2}
Given $k_1,k_2,A,B,C_1,C_2,$ and $D$ as above, we have
$$\sum_{f\in \T \mathcal C(\omega)^{(k_1,k_2,A,B,C_1,C_2,D)}} \mu(f)=
N_{1,n-1}(l'^1_{i_1},\ldots, l'^1_{i_{\gamma_1}},l'_A,l'_B, l'_{k_1}, 
 l'^1_0)
N_{1,n-1}(l'^2_{j_1},\ldots, l'^2_{j_{\gamma_2}}, l'_{k_2}, 
 l'^2_0).$$
\end{prop}
Let us give  a heuristic of the proof of Proposition \ref{floor d=2 2}.
We define $\omega'^1=\{\widetilde L'_A,
\widetilde L'_B,\widehat L'_{k_1},L'_j: j \in C_1\}$,  
$\omega'^2=\{\widehat L'_{k_2},L'_j: j \in C_2\}$, and 
$\omega^{(k_1,k_2,A,B,C_1,C_2,D)}=(\widetilde L'_D,\omega'^1,\omega'^2)$.
As in section \ref{d=1}, 
there is a natural and bijective map
$$\phi:\T \mathcal C(\omega)^{(k_1,k_2,A,B,C_1,C_2,D)}\to 
\T \mathcal C_{red}(\omega^{(k_1,k_2,A,B,C_1,C_2,D)})$$
and Proposition \ref{floor d=2 2}  now follows from
Lemma \ref{complex red}.

\vspace{1ex}
Once again, since it is trivial that $N_{2,1}(2)=1$ and $N_{2,2}(5)=1$, all 
the numbers $N_{n,2}$ can  be computed inductively using Propositions
\ref{floor d=1}, \ref{floor d=2 1}, and \ref{floor d=2 2}.

\begin{exa}(see \cite{Br7}) In Figure  \ref{n=3,d=2}, we depict all possible 
floor decompositions for tropical conics in $\R^3$ passing through a
$(2,3)$-decomposing configuration 
of 8
tropical lines. In each case, we precise the number of $(k_1,k_2,A,B,C_1,C_2,D)$ or 
$(k,A,B)$ with a non-empty corresponding set of tropical morphisms, and the sum
of the multiplicities  of the
corresponding curves in $\T \mathcal
C(\omega)^{(k_1,k_2,A,B,C_1,C_2,D)} $ or 
$\T \mathcal C(\omega)^{(k,A,B)}$ for each such choice.
\end{exa}
\begin{figure}[h]
\centering
\begin{tabular}{ccccccccccccc}
\includegraphics[width=1.0cm, angle=0]{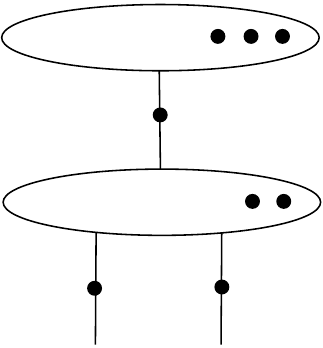}&
 \hspace{2ex}  &
\includegraphics[width=1.0cm, angle=0]{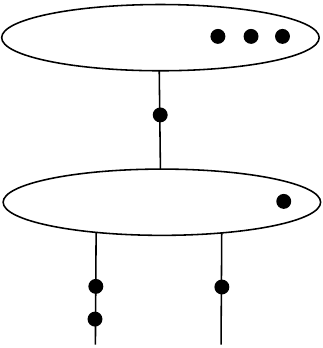}&
 \hspace{2ex}  &
\includegraphics[width=1.0cm, angle=0]{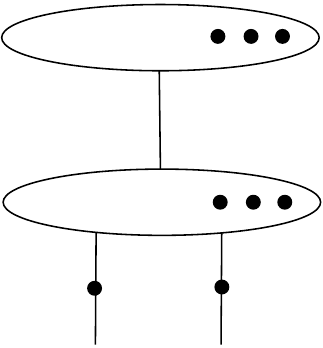}&
 \hspace{2ex}  &
\includegraphics[width=1.0cm, angle=0]{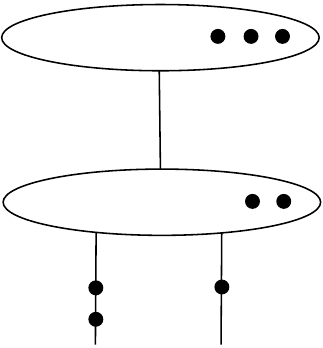}&
 \hspace{2ex}  &
\includegraphics[width=1.0cm, angle=0]{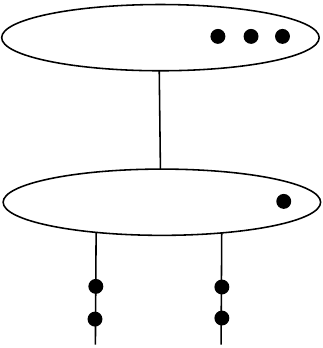}&
 \hspace{2ex}  &
\includegraphics[width=1.0cm, angle=0]{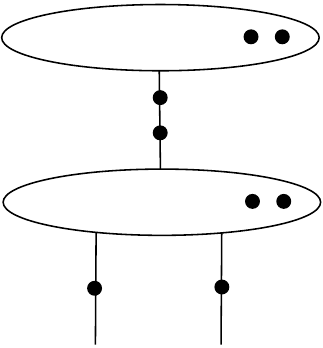}&
 \hspace{2ex}  &
\includegraphics[width=1.0cm, angle=0]{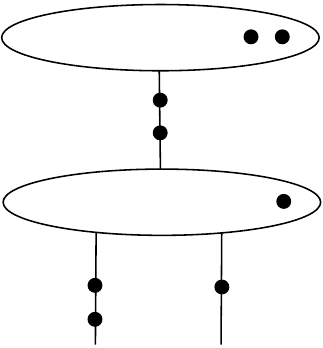}
\\  5, $\mu =1$ &&   3, $\mu =1$
  &&   10, $\mu =1$
&&12,  $\mu =1$
&&3,  $\mu=1$
&&5,  $\mu=1$
&&3,  $\mu=1$
\end{tabular}

\vspace{2ex}

\begin{tabular}{ccccccccccc}
\includegraphics[width=1.0cm, angle=0]{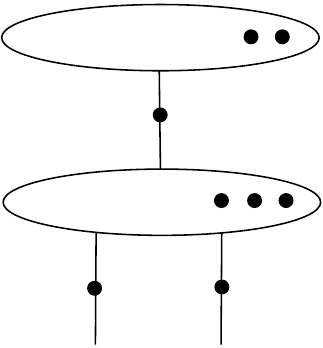}&
 \hspace{2ex}  &
\includegraphics[width=1.0cm, angle=0]{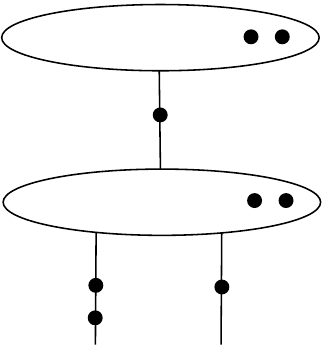}&
 \hspace{2ex}  &
\includegraphics[width=1.0cm, angle=0]{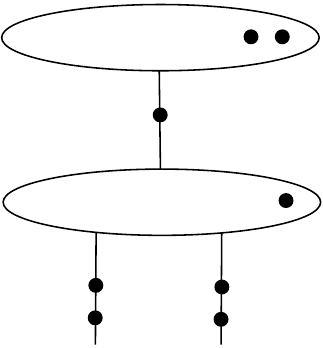}&
 \hspace{2ex}  &
\includegraphics[width=1.0cm, angle=0]{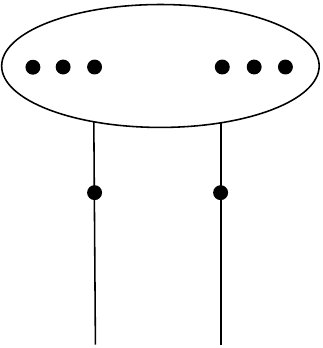}&
 \hspace{2ex}  &
\includegraphics[width=1.0cm, angle=0]{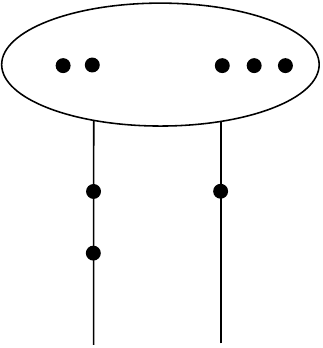}&
 \hspace{2ex}  &
\includegraphics[width=1.0cm, angle=0]{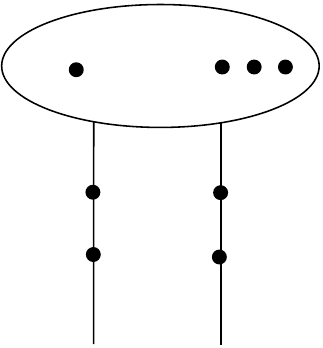}
\\  10, $\mu=1$ &&   12, $\mu=1$
  &&   3, $\mu=1$
&&1,  $\mu=8$
&&3,  $\mu=4$
&&3,  $\mu=2$
\end{tabular}
\caption{Floor decompositions of the 92 tropical conics passing through 8 lines in $\R^3$}
\label{n=3,d=2}
\end{figure}

\section{Maximal configurations for conics}\label{section: proof}

\subsection{Well-ordered totally decomposing configurations}
To prove Theorem \ref{main}, we exhibit \textit{maximal}
configurations, that is configurations $\omega$ of complete
tropical linear spaces such that the cardinal of the set 
$|\T\mathcal C(\omega)| $ is equal to the corresponding Gromov-Witten
invariant.

In this section, we define \textit{well-ordered totally decomposing
  configurations}. The rest of the paper will be devoted to show that any
such configuration is maximal when dealing with conics.

\begin{defn}\label{tot dec}
Let  $\omega= \{L_1,\dots,L_{\gamma}\}$ be a $(d,n)$-decomposing
configuration of complete
tropical linear spaces in $\R^n$, and let $W_i$ be the wall of $L_i$.  

We say that
$\omega$ is a \textit{$(d,n)$-totally decomposing configuration} if  it 
satisfies  one of the two following conditions 
\begin{itemize}
\item $n=2$;
\item for any subset $\Gamma\subset\{1,\ldots,\gamma \}$ the
  configuration $\{\pi(L_i),\ \pi(W_j)\ :\ i\in\Gamma, j\notin\Gamma\}$
 is $(d,n-1)$-totally decomposing.
\end{itemize}
\end{defn}

Since a hyperplane in $\R^n$ has a single vertex, the existence of
totally decomposing configurations of tropical hyperplanes is
straightforward.
Now suppose that we want to construct a totally decomposing configuration
$\omega= \{L_1,\dots,L_{\gamma}\}$ with $\codim L_i=l_i$. We start
with a  totally decomposing configuration of tropical hyperplanes
$\{H_1,\ldots, H_\gamma\}$, and we construct $L_i$ by intersecting
$l_i$ copies of $H_i$  translated along very small vectors.

\begin{rem}
Let $\{L_1,\dots,L_{\gamma}\}$ be a generic totally decomposing configuration
of complete tropical linear spaces,
$\Gamma\subset\{1,\ldots,\gamma \}$, and $\omega'$ the 
  configuration $\{\pi(L_i),\ \pi(W_j)\ :\ i\in\Gamma,
  j\notin\Gamma\}$. Then, it follows directly from Definition \ref{tot dec}
  that given any elements $\mathcal L_1,\ldots,\mathcal L_k$ of $\omega'$, the
  floor of the complete tropical linear spaces $\cap_{i=1}^k\mathcal
  L_k$ is contained in the floor of the lowest space among 
 $\mathcal L_1,\ldots,\mathcal L_k$ (see Figure \ref{conf_index}).
\begin{figure}[h]
\centering
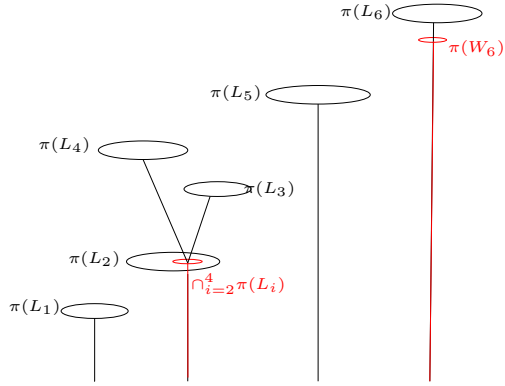
\caption{$\omega=(L_1,\dots,L_6)$}
\label{conf_index}
\end{figure}
\end{rem}

\vspace{1ex}
Next lemma provides an alternate proof of Sottile's Theorem about
maximality of real enumerative problems concerning lines in projective
spaces. 
\begin{lem}\label{lines}
Let $\omega$ be a $(1,n)$-totally decomposing configuration of
complete tropical linear spaces in $\R^n$ subject to
equality (\ref{RR}) with $d=1$. Then $\omega$ is maximal.
\end{lem}

\begin{proof}
We prove the lemma by induction on $n$. Clearly, the lemma is true for
$n=2$. Suppose now that $n\ge 3$ and that the lemma is true in
dimension $n-1$. 
In what follows, we use the notation of section \ref{d=1}.
Any projected configuration
$\omega'^{(k)}$ in $\R^{n-1}$ is  $(1,n-1)$-totally decomposing, 
and so maximal by  induction hypothesis. Hence any
tropical morphism $f$ in $\T\mathcal{C}(\omega'^{(k)})$ has
multiplicity 1.
Since the two sets $\T\mathcal{C}(\omega)^{(k)}$ and 
$\T\mathcal{C}(\omega'^{(k)})$ have the same cardinal, 
we deduce from Proposition \ref{floor d=1} that
$$\sum_{f \in \T
\mathcal{C}(\omega)}\mu(f)=\sum_{k=2}^{\gamma}|\T\mathcal
C(\omega)^{(k)}|=|\T \mathcal{C}(\omega)|.$$
In other words, the configuration $\omega$ is maximal.  
\end{proof}

It turns out that $(2,n)$-totally decomposing configurations are not
necessarily maximal.

\begin{defn}\label{well ord}
Let $\omega=\{L_1,\dots,L_{\gamma}\}$ be a $(d,n)$-totally decomposing
configuration of complete tropical linear spaces in $\R^n$, and let $W_i$ be
the wall of $L_i$.
We say that $\omega$ is a well-ordered  $(d,n)$-totally decomposing
configuration if it satisfies one of the two following conditions 
\begin{itemize}
 \item  $n=2$;
 \item for any subset $\Gamma\subset\{1,\ldots,\gamma \}$ the
  configuration $\{\pi(L_i),\ \pi(W_j)\ :\ i\in\Gamma, j\notin\Gamma\}$
 is a well-ordered $(d,n-1)$-totally decomposing configuration; 
moreover for any $i$ and $j$ such that
 $L_i\gg L_j$, we have $\pi(L_i)\gg 
\pi(L_j)$ and $\pi(W_i)\gg
\pi(L_j)$ if $\codim \pi(L_j) \geq 1$, $\pi(L_i)\gg\pi(W_j)$ and $\pi(W_i)\gg\pi(W_j)$.
\end{itemize}
\end{defn}

Again it is trivial that well-ordered totally decomposing
configurations of hyperplanes exist, from which it follows 
that there exists a
 well ordered totally decomposing
configurations $\omega= \{L_1,\dots,L_{\gamma}\}$ with $\codim
L_i=l_i$ for any fixed positive integers $l_1,\dots,l_{\gamma}$.

\begin{rem}
Let $\{L_1,\dots,L_{\gamma}\}$ be a generic well-ordered 
totally decomposing configuration
of complete tropical linear spaces,
$\Gamma\subset\{1,\ldots,\gamma \}$, and $\omega'$ the 
  configuration $\{\pi(L_i),\ \pi(W_j)\ :\ i\in\Gamma,
  j\notin\Gamma\}$. Then, it follows directly from Definition
  \ref{well ord}
  that given any elements $\mathcal L_1,\ldots,\mathcal L_k$ of
  $\omega'$, 
the configuration obtained out of $\omega'$ by replacing 
$\mathcal L_1,\ldots,\mathcal L_k$ by $\cap_{i=1}^k\mathcal
  L_k$ is still
a well-ordered 
totally decomposing configuration (see Figure \ref{conf_index}, where $L_6\gg\dots\gg L_1$).
\end{rem}

We will prove in Proposition \ref{conics} that a well-ordered
$(2,n)$-totally decomposing configuration is maximal. 
We will treat the case of tropical conics with two floors of degree
one using the totally decomposing hypothesis more or less as in Lemma
\ref{lines} (see Proposition \ref{card red}), and we will
treat the case of tropical conics with
one floor of degree 2
using the well-ordered hypothesis.

Before proving Proposition \ref{conics} in its full generality, let us first
 illustrate on a simple example how the 
well-ordered hypothesis  solves the case 
of tropical conics with
one floor of degree 2.

\subsection{Conics through 8 lines in $\R^3$}\label{conics8lines} 
Maximality of the
real enumerative problem concerning conics passing through $8$ spatial
lines in general position in $\R P^3$
was first announced by Brugallé and Mikhalkin
in \cite{Br7}. 

Let us fix  a well-ordered $(2,3)$-totally decomposing
configuration $\omega=\{L_1,\ldots,L_8\}$ of 8 tropical lines in
$\R^3$. We suppose that $L_8\gg L_7\gg\ldots\gg L_1$, and we denote by
$W_i$ the wall of $L_i$.
Recall that notations have been defined in
section \ref{d=2}.

\begin{lem}
If the triple $(k,A,B)$ is such that 
$\T\mathcal C(\omega)^{(k,A,B)}$ is non-empty, then it contains 
exactly $2^{m_{(k,A,B)}}$ tropical morphisms, all of them of
multiplicity 1.
\end{lem}
\begin{proof}
Let us fix such a triple $(k,A,B)$.
It is easy to see that
the three higher lines $L_6$, $L_7$ and
$L_8$ are vertical constraints for elements of 
$\T\mathcal C(\omega)^{(k,A,B)}$ (see Figure \ref{n=3,d=2}).
 In particular, the configuration
$\omega'^{(k,A,B)}$ in $\R^2$ contains the 3 points $w_8=\pi(W_8)$,
$w_7=\pi(W_7)$, and $w_6=\pi(W_6)$. Since $\omega$ is a well-ordered
$(2,3)$-totally decomposing configuration, the configuration 
$\omega'^{(k,A,B)}$ is $(2,2)$-decomposing and
the points $w_8,w_7$, and $w_6$ are
its highest elements. Moreover, $N_{2,2}(2,2,2,2,2)=1$ so 
$\omega'^{(k,A,B)}$ is maximal.
Hence the set $\T \mathcal C(\omega'^{(k,A,B)})$
reduces to a unique plane tropical conic $f:C\to \R^2$ 
(passing through the three points $w_8$,
$w_7$, $w_6$, and two
other ones) which intersects the $m_{(k,A,B)}=1,\ 2$
or $3$ tropical lines in $\omega'^{(k,A,B)}$.

It remains  to show  that each of these $m_{(k,A,B)}$ tropical lines
intersects  the tropical conic
$f(C)$
 in two distinct points, which would imply the lemma by Proposition
\ref{floor d=2 1}. According to what we discussed above about the
configuration
$\omega'^{(k,A,B)}$, the tropical
 conic $f$ has two floors of degree 1 passing through the points $w_8$
 and $w_6$, while the bounded elevator of $f$ passes through $w_7$ 
(see figure \ref{plan_con}). Since any tropical line in 
$\omega'^{(k,A,B)}$ is lower than $w_6$, it must intersect $f(C)$ 
along its unbounded elevators. In particular, it intersects $f(C)$ in
two distinct points.
\end{proof}
 
\begin{figure}[h]
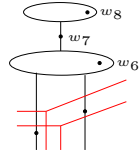
\caption{Floor decomposition of a planar conic}\label{plan_con}
\end{figure}

It is immediate that the set $\T\mathcal C(\omega)^{(k_1,k_2,A,B,C_1,C_2,D)}$ is
either empty of composed of a unique tropical morphism of multiplicity
1 (see Figure \ref{n=3,d=2}). So the set $\T\mathcal C(\omega)$ is composed of tropical
morphisms of multiplicity 1. 
Hence according to 
Lemma \ref{real mod},  there exists a 
generic configuration of $8$ lines in $\RP^3$
such that exactly $92$ real conics intersect these 8 lines.

\subsection{General case}

Theorem \ref{main} is a direct consequence of 
Proposition \ref{conics}, Theorem \ref{corr}, and Lemma \ref{real mod}.
Next proposition relies on Proposition \ref{card red} which is proved
in next section.

\begin{prop}\label{conics}
Let $\omega$ be a well-ordered $(2,n)$-totally decomposing
configuration 
of complete tropical linear spaces in $\R^n$ subject to
equality (\ref{RR}). Then $\omega$ is maximal.
\end{prop}

\begin{proof}
Our goal is to prove that $\sum_{f\in \T\mathcal
  C(\omega)}\mu(f)=|\T\mathcal C(\omega)|$. This is obviously true
when $n=2$ since the left hand side is $1$. Suppose that $n\geq 3$ and
the equality is true in lower dimensions. Since $\omega$ is
$(2,n)$-decomposing, we have $$\sum_{f\in \T\mathcal C(\omega)}\mu(f)
=\sum_{(k,A,B)} \left(\sum_{f\in \T\mathcal
  C(\omega)^{(k,A,B)}}\mu(f)\right)+\sum_{(k_1,k_2,A,B,C_1,C_2,D)}
\left(\sum_{f\in \T\mathcal
  C(\omega)^{(k_1,k_2,A,B,C_1,C_2,D)}}\mu(f)\right)$$ 
where the sums
are taken following section \ref{d=2}.
\vspace{1ex}
Suppose that $\T\mathcal C (\omega)^{(k,A,B)}$ is non-empty.
By
 induction, 
$\omega'^{(k,A,B)}$ is a well-ordered $(2,n-1)$-totally decomposing
configuration and so is
maximal.
Moreover,
the map $\phi_{(k,A,B)}:\T\mathcal C (\omega)^{(k,A,B)}\to
\T\mathcal C(\omega'^{(k,A,B)})$  is $2^{m(k,A,B)}$ to $1$, counted with
multiplicities. 
Hence, it remains  to show  that any
 conic in $\T\mathcal C(\omega'^{(k,A,B)})$ 
intersects each of the $m(k,A,B)$ hyperplanes
of $\omega'^{(k,A,B)}$ in 2 distinct points, which would imply,
by Proposition
\ref{floor d=2 1}, that
$$|\T\mathcal
C(\omega)^{(k,A,B)}|=2^{m(k,A,B)}|\T\mathcal C(\omega'^{(k,A,B)})|=
2^{m(k,A,B)}\sum_{f'\in \T\mathcal
  C(\omega'^{(k,A,B)})}\mu(f') =\sum_{f\in \T\mathcal
  C(\omega)^{(k,A,B)}}\mu(f).$$
 
We have $\omega'^{(k,A,B)}=\{\widetilde L'_A,\widetilde L'_B,\widehat
L'_k, L'_{k+1},\dots,L'_{\gamma}\}$, 
so a tropical hyperplane $H$ in  $\omega'^{(k,A,B)}$ is either $\widehat L'_k$, $\widetilde L'_A$ or
$\widetilde L'_B$. 
Since $\omega'^{(k,A,B)}$ is a well-ordered configuration,  
$\widehat L'_k$, $\widetilde L'_A$ and $\widetilde L'_B$ are its 3 lowest elements. In particular,
$H$ is higher than at most 2 other elements of $\omega'^{(k,A,B)}$.
 
The configuration $\omega'^{(k,A,B)}$ is $(2,n-1)$-decomposing, so 
 any unbounded elevator and any floor of a conic in $\T\mathcal C
(\omega'^{(k,A,B)})$ has to intersect at least one element of 
 $\omega'^{(k,A,B)}$ which is not a hyperplane. In particular,
a hyperplane in $\omega'^{(k,A,B)}$  intersects
 such a conic strictly below the lowest floor, that is along its two
unbounded elevators in two distinct points (see figure
\ref{One_two_floor_dec}).

\begin{figure}[h]
\includegraphics[height=1.5cm]{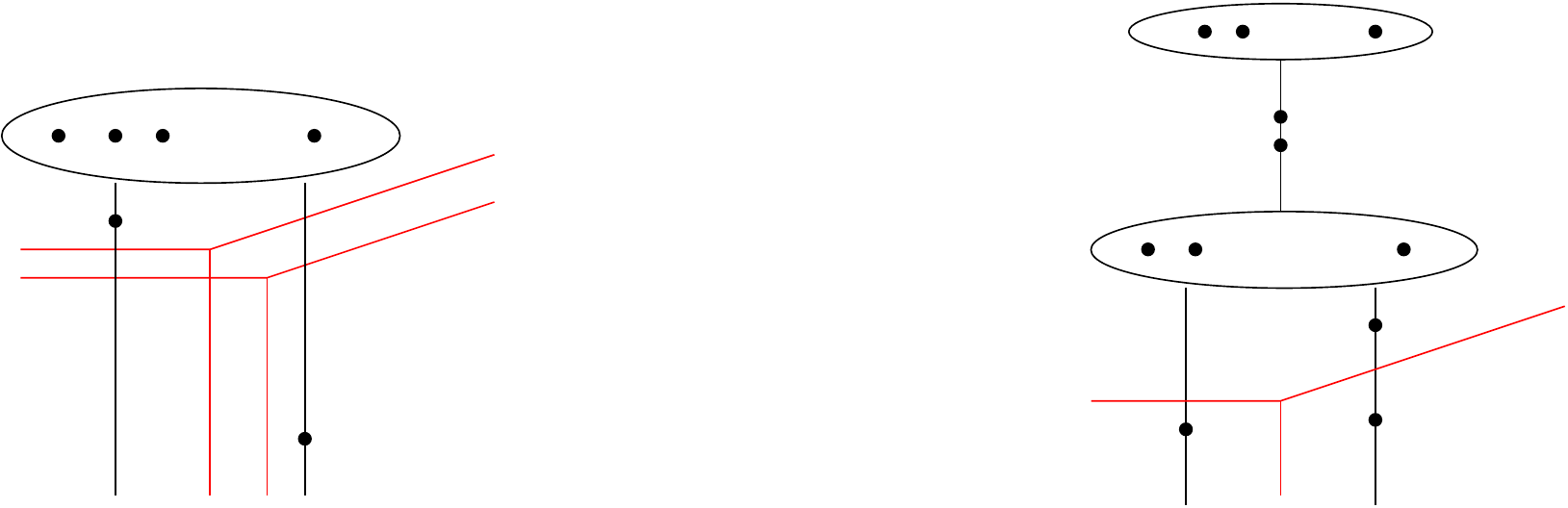}
\caption{Floor decompositions in $\T\mathcal C(\omega'^{(k,A,B)})$ and hyperplanes.}\label{One_two_floor_dec}
\end{figure}

\vspace{1ex}

Suppose that $\T\mathcal C(\omega)^{(k_1,k_2,A,B,C_1,C_2,D)}$ is non-empty.
According to section \ref{d=2}, we have
\begin{equation}\label{ineq}
\begin{array}{lll}
|\T \mathcal C_{red}(\omega^{(k_1,k_2,A,B,C_1,C_2,D)})|&=&
|\T\mathcal C(\omega)^{(k_1,k_2,A,B,C_1,C_2,D)}|
\\ \\&\le& 
\sum_{f\in \T\mathcal C(\omega)^{(k_1,k_2,A,B,C_1,C_2,D)}}\mu(f)
\\ \\&\le&
N_{1,n-1}(l'^1_{i_1},\ldots, l'^1_{i_{\gamma_1}},l'_A,l'_B, l'_{k_1}, 
 l'^1_0)
N_{1,n-1}(l'^2_{j_1},\ldots, l'^2_{j_{\gamma_2}}, l'_{k_2}, 
 l'^2_0).
\end{array}
\end{equation}

According to Proposition \ref{card red} (see next section) 
applied to
$\omega^{(k_1,k_2,A,B,C_1,C_2,D)}=(\widetilde L'_D,
\omega'^1,\omega'^2) $, we get that
$$|\T \mathcal C_{red}(\omega^{(k_1,k_2,A,B,C_1,C_2,D)})|=
N_{1,n-1}(l'^1_{i_1},\ldots, l'^1_{i_{\gamma_1}},l'_A,l'_B, l'_{k_1}, 
 l'^1_0)
N_{1,n-1}(l'^2_{j_1},\ldots, l'^2_{j_{\gamma_2}}, l'_{k_2}, 
 l'^2_0). $$
Hence all inequalities in (\ref{ineq}) are in fact equalities, and we
have
$$|\T\mathcal C(\omega)^{(k_1,k_2,A,B,C_1,C_2,D)}|=
\sum_{f\in \T\mathcal C(\omega)^{(k_1,k_2,A,B,C_1,C_2,D)}}\mu(f)$$ which achieves to prove the proposition.
\end{proof}

\subsection{Reducible conics through a well-ordered totally
  decomposing configuration}
 A reducible tropical morphism  $f:C\to \R^n$ of degree 2 has either
two floors of degree 1, or a unique (reducible) floor of degree
2. 
The proof of Proposition \ref{floor horizontal} only relies on the 
 finiteness of  the set 
$\T \mathcal C(\omega)$ and on vectors $u_{f,e}$ for
 $e\in\Ed^\infty(C)$. In particular Lemma \ref{generic red} implies
 that
Proposition \ref{floor horizontal} 
still holds for elements of 
$\T \mathcal C_{red}(L_0,\omega^1,\omega^2)$.
In this section we
 compute the numbers of reducible tropical morphisms of degree 2
 passing through a particular configuration of complete tropical
 linear spaces in $\R^n$.

\vspace{1ex}
Let  
$l_0\ge 0$ and $l^1_1$, $\ldots$, $l^1_{\gamma_1},l^2_1$, $\ldots$,
$l^2_{\gamma_2}\ge 1$ be some integers such that 
$$l_0+\sum_{i=1,2}\quad \sum_{j=1}^{\gamma_i}(l^i_j-1)=3n-2.$$
We choose $L_0$ a complete linear space in $\R^n$
of codimension $l_0$ and
two configurations $\omega^1$ and $\omega^2$ of 
complete  tropical linear spaces in $\R^n$ such that 
$\omega^i=\{L^i_1,\ldots,L^i_{\gamma_i}\}$ with
 $\codim L^i_j=l^i_j$ and $L^i_j\gg L^i_{j+1}$.
\begin{defn}
The configuration $(L_0, \omega^1, \omega^2)$ is said to be
separated if the configuration
$\{L_0\}\cup\omega^1\cup \omega^2$   is a well-ordered
$(2,n)$-totally decomposing configuration,  and if $L_0$ and elements of
$\omega^1$  are below elements of $\omega^2$.
\end{defn}
 Note that we do not make any
assumption
 about the mutual  position of $L_0$ and elements of
$\omega^1$, and that the projection to $\R^{n-1}$ of a separated
 configuration in $\R^n$ is still separated.
Given a separated configuration $\{L_0, \omega^1, \omega^2\}$ in $
\R^n$, 
we denote by 
$\T N^{red}_n(l_0,\{l^1_1,\ldots,l^1_{\gamma_1}\},\{l^2_1,\ldots,l^2_{\gamma_2}\})$
the cardinal of $\T \mathcal C_{red}(L_0,\omega^1,\omega^2)$. 
We will see in Proposition \ref{card red} 
that this cardinal does not depend on $L_0,\ \omega^1$, and
$\omega^2$ as long as $\{L_0, \omega^1, \omega^2\}$ is separated.

\begin{prop}\label{card red}
For any generic separated configuration $\{L_0, \omega^1, \omega^2\}$ in 
$\R^n$, we have
$$\T N^{red}_n(l_0,\{l^1_1,\ldots,l^1_{\gamma_1}\},\{l^2_1,\ldots,l^2_{\gamma_2}\})=
\prod_{i=1,2}N_{1,n}\left( 2n-1
-\sum_{j=1}^{\gamma_i}(l^i_j-1),l^i_1,\ldots,l^i_{\gamma_i}\right)
.$$
\end{prop}
\begin{proof}
The case $n=2$ is straightforward. Let us suppose now that $n\ge 3$
and that the Proposition is true in lower dimensions.

Let $f:C_1\cup_p C_2\to\R^n$ be an element of $\T \mathcal
C_{red}(L_0,\omega^1,\omega^2)$. 
An elevator of $f$ has to meet at least a constraint, and 
elements of $\omega^2$ are above $L_0$ and elements of $\omega^1$, so 
 $f$
 has two floors of degree 1; moreover the node of $C$ is
either on both elevators of $C_1$ and $C_2$, or on the floor of $C_1$
and the elevator of $C_2$ (see figure \ref{red_dec}).
\begin{figure}[h]
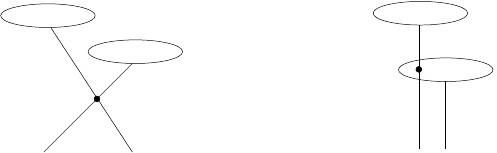
\caption{}\label{red_dec}
\end{figure}

We define $k_0$ as the smallest integer such that $L^1_{k_0}$ is
higher than
$L_0$ if such an element of $\omega^1$ exists, and by $k_0=0$
otherwise. We denote by
$\mathcal F_i$ the floor of $C_i$, $i=1,2$.

We consider the partition 
$$\bigsqcup_{\begin{array}{l}k_0\le k_1\le \gamma_1\\1\le k_2\le \gamma_2\end{array}
}\mathcal C_1^{k_1,k_2}\ 
\bigsqcup_{\begin{array}{l}1\le k_2\le \gamma_2\end{array}}
\mathcal C_2^{k_2}
\ \bigsqcup_{\begin{array}{l}1\le k_1\le k_0-1\\1\le k_2\le \gamma_2\end{array}
}
\mathcal C_3^{k_1,k_2} $$
of the set  
$\T \mathcal C_{red}(L_0,\omega^1,\omega^2)$ where\\
\begin{itemize}
\item $\mathcal C_1^{k_1,k_2}$ is the set of all elements 
$f_1\cup_pf_2:C_1\cup_p C_2\to\R^n$ in 
$\T \mathcal C_{red}(L_0,\omega^1,\omega^2)$
such  that $p$ is on both elevators of
  $f_1$ and $f_2$, and the floor   $\F_i$ meets the horizontal
constraint $L^i_{k_i}$; note that 
$\mathcal C_1^{k_1,k_2}=\emptyset$ if $k_0=0$;

\item $\mathcal C_2^{k_2}$ is the set of all elements 
$f_1\cup_pf_2:C_1\cup_p C_2\to\R^n$ in 
$\T \mathcal C_{red}(L_0,\omega^1,\omega^2)$
such that  $p$ is on  the floor $\F_1$,  $ L_0$ is the horizontal
constraint of $\F_1$, 
and the floor $\F_2$ meets the horizontal
constraint $L^2_{k_2}$; note that 
$\mathcal C_2^{k_2}=\emptyset$ if $k_0\le 1$;

\item  $\mathcal C_3^{k_1,k_2}$ is the set of all elements 
$f_1\cup_pf_2:C_1\cup_p C_2\to\R^n$ in 
$\T \mathcal C_{red}(L_0,\omega^1,\omega^2)$
such that $p$ is on the floor 
  $\F_1$, and the floor   $\F_i$ meets the horizontal
constraint $L^i_{k_i}$.
\end{itemize}
We denote by $W^i_j$ (resp. $W_0$)
the wall of the constraint $L^i_j$ (resp. $L_0$). We consider the
following complete tropical linear spaces in $\R^{n-1}$
$$L'^i_j=\pi(W^i_j),\quad  \quad 
 \widehat L'^i_j=\pi(L^i_j), \quad \quad 
 \widetilde L'^{k_1,k_2}_0=\pi(L_0)\bigcap_{i=1,2}\ \ \bigcap_{1\le j\le
   k_i-1} \pi(L^i_j),$$
$$\widetilde 
L'^1_{k}=\bigcap_{1\le j\le k-1}\pi(L^1_j),\quad \quad   
\widetilde L'^{k_2}_0=\pi(L_0)\bigcap_{1\le j\le
   k_2-1} \pi(L^2_j), \quad \quad 
\mathcal L'^{k_2}_0=\pi(W_0)\bigcap_{1\le j\le
   k_2-1} \pi(L^2_j)$$
and the following configurations
$$\omega_1^{i,k_i}=\{\widehat L'^i_{k_i},
L'^i_{k_i+1},\ldots,L'^i_{\gamma_i}\} \quad \text{for}\ i=1,2, \quad
\omega_2^{1}=\{\widetilde L'^1_{k_0},L'^1_{k_0},\ldots,L'^1_{\gamma_1}\},\quad  
\omega_3^{1,k_1}=\{\widehat L'^1_{k_1}, \widetilde L'^1_{k_1}, L'^1_{k_1+1},\ldots,L'^1_{\gamma_1}\}.$$

\vspace{1ex}
Given an element $f$ of $\mathcal C_1^{k_1,k_2}$, the tropical
morphism $\pi\circ f$ induces an element of 
$\T \mathcal C_{red}(\widetilde
L'^{k_1,k_2}_0,\omega_1^{1,k_1},\omega_1^{2,k_2})$. 
Conversely any
element $f'_1\cup_{p'} f'_2:C'_1\cup_{p'} C'_2\to\R^{n-1}$
 of $\T \mathcal C_{red}(\widetilde
L'^{k_1,k_2}_0,\omega_1^{1,k_1},\omega_1^{2,k_2})$ has a unique lift
  $f_1\cup_{p} f_2:C_1\cup_{p} C_2\to\R^{n}$ in $\mathcal C_1^{k_1,k_2}$; the
elevators of $f_1$ and $f_2$ correspond to $p'$, and the
node $p$ is at the unique intersection point of the elevator of
$C_i$ with $L_0$. Hence, we get that the total number of tropical
morphisms $f$ in $\mathcal C_1^{k_1,k_2}$ is
$$
\T N^{red}_{n-1}\left(l_0-1+
\sum_{i=1,2}\sum_{j=1}^{k_i-1}(l^i_j-1),\ 
\{l^1_{k_1}-1,l^1_{k_1+1},\ldots,l^1_{\gamma_1}\}, \
\{l^2_{k_2}-1,l^2_{k_2+1},\ldots,l^2_{\gamma_2}\}
\right).$$

\vspace{1ex}
Given an element $f$ of $\mathcal C_2^{k_2}$, the tropical
morphism $\pi\circ f$ induces an element of 
$\T \mathcal C_{red}(\widetilde
L'^{k_2}_0,\omega_2^{1},\omega_1^{2,k_2})$. 
Conversely any
element $f'_1\cup_{p'} f'_2:C'_1\cup_{p'} C'_2\to\R^{n-1}$
 of $\T \mathcal C_{red}(\widetilde
L'^{k_2}_0,\omega_2^{1},\omega_1^{2,k_2})$ has a unique lift
  $f_1\cup_{p} f_2:C_1\cup_{p} C_2\to\R^{n}$ in $\mathcal C_2^{k_2}$;
 the
elevator of $f_2$ corresponds to the node  $p'$,  the
elevator of $f_1$ corresponds to the unique intersection point of
$C'_1$ and  $\widetilde L'^1_{k_0} $, and the node $p$
 corresponds to the unique intersection point of the elevator of
$f_2$ and $L_0$.
 Hence, we get that the total number of  tropical
morphisms $f$ in  $\mathcal C_2^{k_2}$ is
$$
 \T N^{red}_{n-1}\left(l_0-1+
\sum_{j=1}^{k_2-1}(l^2_j-1),\ 
\{ \sum_{j=1}^{k_0-1}(l^1_j-1),l^1_{k_0},\ldots,l^1_{\gamma_1}\},
\ \{l^2_{k_2}-1,l^2_{k_2+1},\ldots,l^2_{\gamma_2}\}
\right).$$

\vspace{1ex}
Given an element $f$ of $\mathcal C_3^{k_1,k_2}$, the tropical
morphism $\pi\circ f$ induces an element of 
$\T \mathcal C_{red}(\mathcal
L'^{k_2}_0,\omega_3^{1,k_1},\omega_1^{2,k_2})$. 
Conversely any
element $f'_1\cup_{p'} f'_2:C'_1\cup_{p'} C'_2\to\R^{n-1}$
 of $\T \mathcal C_{red}(\mathcal
L'^{k_2}_0,\omega_3^{1,k_1},\omega_1^{2,k_2})$ has a unique lift
  $f_1\cup_{p} f_2:C_1\cup_{p} C_2\to\R^{n}$ in $\mathcal
C_3^{k_1,k_2}$.
 Hence, we get that the total number of tropical
morphisms $f$ in $\mathcal C_3^{k_1,k_2}$  is
$$\T N^{red}_{n-1}\left(l_0+
\sum_{j=1}^{k_2-1}(l^2_j-1),\ 
\{\sum_{j=1}^{k_1-1}(l^1_j-1),l^1_{k_1}-1,l^1_{k_1+1},\ldots,l^1_{\gamma_1}\},
\ \{l^2_{k_2}-1,l^2_{k_2+1},\ldots,l^2_{\gamma_2}\}
\right).$$

\vspace{2ex}
Altogether with the induction hypothesis, we get that
$$\T N^{red}_n(l_0,\{l^1_1,\ldots,l^1_{\gamma_1}\},\{l^2_1,\ldots,l^2_{\gamma_2}\})=
AB$$
where
$$A=\sum_{k_2=1}^{\gamma_2}
N_{1,n-1}(2n-2 -\sum_{j=k_2}^{\gamma_2}(l^2_j-1), \ 
l^2_{k_2}-1,l^2_{k_2+1},\ldots,l^2_{\gamma_2})$$
and 
$$\begin{array}{lll}
B&=&\sum_{k_1=k_0}^{\gamma_1}
N_{1,n-1}(2n-2 -\sum_{j=k_1}^{\gamma_1}(l^1_j-1),\ 
l^1_{k_1}-1,l^1_{k_1+1},\ldots,l^1_{\gamma_1}) 
\\ \\ && + 
N_{1,n-1}(2n-2- \sum_{j=1}^{\gamma_1}(l^1_j-1) ,\
\sum_{j=1}^{k_0-1}(l^1_j-1),\
l^1_{k_0},\ldots,l^1_{\gamma_1}) 
\\ \\ && + 
\sum_{k_1=1}^{k_0-1}
N_{1,n-1}(2n-1 -\sum_{j=1}^{\gamma_1}(l^1_j-1),\
\sum_{j=1}^{k_1-1}(l^1_j-1),\
l^1_{k_1}-1,l^1_{k_1+1},\ldots,l^1_{\gamma_1})) 
\end{array}$$

\vspace{1ex}
Now it follows from Corollary \ref{rec line} that
$$A= N_{1,n}( 2n-1- \sum_{j=1}^{\gamma_2}(l^2_j-1),l^2_1,\ldots, l^2_{\gamma_2})
$$
and that
$$B=N_{1,n}(2n-1
-\sum_{j=1}^{\gamma_1}(l^1_j-1),l^1_1,\ldots,l^1_{\gamma_1} ) $$
so the Proposition is proved.
\end{proof}

\small

\def\rightmark{\em Bibliography}

\bibliographystyle{alpha}

\bibliography{../Biblio}

\end{document}